\documentclass[a4paper,10pt]{amsart}

\usepackage{amssymb,amsmath,amsfonts,amsthm,mathtools}
\usepackage{latexsym,mathrsfs,pb-diagram}
\usepackage[pdftex]{graphicx}
\usepackage[all]{xy}
\usepackage[hmargin=2.5cm,vmargin=3.5cm]{geometry}
\usepackage[plainpages=false,colorlinks,pdfpagelabels]{hyperref}
\hypersetup{
  urlcolor=black,
  citecolor=green,
  linkcolor=blue
}

\numberwithin{equation}{section}

\theoremstyle{definition}
\newtheorem{Def}{Definition}[section]

\theoremstyle{remark}
\newtheorem{Exa}[Def]{Example}
\newtheorem{Rem}[Def]{Remark}

\theoremstyle{plain}
\newtheorem{Prop}[Def]{Proposition}
\newtheorem{Cor}[Def]{Corollary}
\newtheorem{Thm}[Def]{Theorem}
\newtheorem{Lem}[Def]{Lemma}

\newcommand{\dfn}{\mathrel{\dot{=}}}

\newcommand{\st}{ \ ; \ }
\newcommand{\rarr}{\rightarrow}
\newcommand{\sset}{\subset}

\newcommand{\Z}{\mathbb{Z}}
\newcommand{\N}{\mathbb{N}}
\newcommand{\R}{\mathbb{R}}

\newcommand{\C}{\mathbb{C}}

\newcommand{\TR}[5]{\begin{array}{c c c c c}
    {#1} & : & {#3} & \longrightarrow & {#5}\\
    & & {#2} & \longmapsto & {#4}
  \end{array}
}

\newcommand{\transp}[1]{\prescript{\mathrm{t}}{}{{#1}}}

\DeclareMathOperator{\ran}{\mathrm{ran}}

\newcommand{\del}{\partial}
\newcommand{\dd}{\mathrm{d}}

\newcommand{\D}{\mathscr{D}}

\newcommand{\cinfty}{\mathscr{C}^\infty}
\newcommand{\gev}{\mathscr{G}}
\newcommand{\sob}{\mathscr{H}}

\DeclareMathOperator{\T}{T}
\newcommand{\MM}{\mathrm{M}}
\newcommand{\LL}{\mathrm{L}}
\newcommand{\VV}{\mathcal{V}}
\newcommand{\psum}{\sideset{}{'}\sum}

\newcommand{\hsum}[1]{\underset{#1}{\widehat{\bigoplus}}}

\DeclareMathOperator{\End}{\mathrm{End}}

\newcommand{\gr}[1]{\mathfrak{#1}}

\DeclareMathOperator{\ad}{\mathrm{ad}}
\newcommand{\vv}[1]{\mathrm{#1}}

\author{Gabriel Ara\'{u}jo}
\address{University of S{\~a}o Paulo, ICMC-USP, S{\~a}o Carlos, SP, Brazil}
\email{\texttt{gccsa@icmc.usp.br}}
\thanks{This work was partially supported by Conselho Nacional de Desenvolvimento Cient{\'i}fico e Tecnol{\'o}gico (CNPq, grant~140838/2012-0) and the S{\~a}o Paulo Research Foundation (FAPESP, grant~2018/12273-5).}

\title[Left-invariant systems on Lie groups]{Global regularity and solvability of left-invariant differential systems on compact Lie groups}

\begin{document}

\begin{abstract} We are interested in global properties of systems of left-invariant differential operators on compact Lie groups: regularity properties, properties on the closedness of the range and finite dimensionality of their cohomology spaces, when acting on various function spaces e.g.~smooth, analytic and Gevrey. Extending the methods of Greenfield and Wallach~\cite{gw73} to systems, we obtain abstract characterizations for these properties and use them to derive some generalizations of results due to Greenfield~\cite{greenfield72}, Greenfield and Wallach~\cite{gw72}, as well as global versions of a result of Caetano and Cordaro~\cite{cc11} for involutive structures.
\end{abstract}

\maketitle


\section{Introduction}
Linear PDOs can act on various function spaces, provided their coefficients are sufficiently regular: smooth, real-analytic and/or Gevrey spaces, as well as their generalized counterparts, just to name a few. It is then of interest to investigate properties of regularity and solvability (either local or global; several flavors of hypoellipticity; properties of the associated cohomology spaces for systems; and so on) of such PDOs in some of these spaces, sometimes providing radically different answers depending on the spaces and the operators under study.

Some results in the literature, however, establish relationships among such properties for whole classes of operators. Greenfield~\cite{greenfield72} (see also his work with Wallach~\cite{gw72}) proved that for operators with constant coefficients on a torus global hypoellipticity implies global analytic-hypoellipticity. Dealing with differential complexes associated to locally integrable structures, Caetano and Cordaro~\cite{cc11} proved that if in a given degree the complex is locally solvable in the smooth setup then it is also locally solvable in the Gevrey setup (same degree), while Ragognette~\cite{ragognette18}, using similar methods, relates these with local solvability in the sense of Gevrey ultradistributions. Still dealing with locally integrable structures, Malaspina and Nicola~\cite{mn14} conjectured another connection between smooth and Gevrey local solvability (a kind of converse to the result of Caetano and Cordaro), which remains open except for a few positive cases. Schapira~\cite{schapira69} and Suzuki~\cite{suzuki72} dealt with solvability in spaces of real-analytic functions and hyperfunctions, from the so-called semi-local viewpoint, and its relationship with condition~$(\mathcal{P})$ of Nirenberg-Treves; this in turn recently inspired Cordaro and the author to investigate similar questions in the framework of locally integrable structures.

In the present work our point of view is global, and we deal with systems of left-invariant differential operators on compact Lie groups (see below), which we now explain.

Let $G$ be a Lie group. The identity element of $G$ will be denoted by $e$, and by $L_x: G \rarr G$ we denote the left translation by $x \in G$: it is a diffeomorphism of $G$ onto itself, and a (real or complex) vector field $\vv{X}$ on $G$ is said to be \emph{left-invariant} if $(L_x)_* \vv{X} = \vv{X}$ for every $x \in G$. This notion extends naturally to any class of globally defined ``objects'' on $G$ that can be pulled back or pushed forward by diffeomorphisms. Lie groups are automatically real-analytic manifolds, and left-invariant vector fields are automatically real-analytic. We denote by $\gr{g}$ the Lie algebra of all \emph{real} vector fields on $G$ that are left-invariant, and by $\C \gr{g}$ its complexification.

A vector subbundle $\VV \sset \C T G$ is called \emph{left-invariant} if
\begin{align*}
  (L_x)_* \VV_y &\sset \VV_{xy}, \quad \forall x,y \in G.
\end{align*}
It is easy to show that the inclusion above cannot be proper; thus left-invariant subbundles of $\C T G$ are in one-to-one correspondence with linear subspaces $\VV_e$ of $\C T_eG$ and hence are automatically real-analytic. If one considers $\C T_eG$ as isomorphic to $\C \gr{g}$ (by means of the map ``evaluation at $e$'') then it is immediate that $\VV$ is involutive -- i.e.~the Lie bracket of any two smooth sections of $\VV$ is again a smooth section of $\VV$ --  if and only if $\VV_e$ corresponds to a complex Lie subalgebra $\gr{v}$ of $\C \gr{g}$ (which is precisely the space of all left-invariant sections of $\VV$). So we might just as well recast the definition above as follows: a left-invariant involutive structure on $G$ is simply a complex Lie subalgebra of $\C \gr{g}$.

Aside from systems of vector fields with constant coefficients on tori, one should mention as an important (although simple) example of such structures the CR structure on the $3$-sphere inherited from its natural embedding into $\C^2$: it is well-known that $\mathbb{S}^3$ can be made into a (non-commutative) Lie group by identifying it with $\mathrm{SU}(2)$ via the diffeomorphism
\begin{align}
  (z_1, z_2) &\longmapsto \left(
    \begin{array}{c c}
      z_1 & -\bar{z}_2 \\
      z_2 & \bar{z}_1
    \end{array}
    \right). \label{eq:id_sphere}
\end{align}
The CR structure on $\mathbb{S}^3$ is then spanned by the single vector field
\begin{align*}
  -z_2 \frac{\del}{\del \bar{z}_1} + z_1 \frac{\del}{\del \bar{z}_2}
\end{align*}
which one checks by hand to be left-invariant; indeed, it corresponds, via the identification~\eqref{eq:id_sphere}, to
\begin{align*}
  \frac{1}{2} \left(
    \begin{array}{c c}
      0 & -1 \\
      1 & 0
    \end{array}
    \right) + \frac{i}{2} \left(
    \begin{array}{c c}
      0 & i \\
      i & 0
    \end{array}
    \right) \in \C \gr{su}(2).
\end{align*}
    
Next we recall what are the differential complexes naturally associated with $\VV$ (cf.~\cite[Section~I.6]{treves_has}): we define a vector subbundle $\T' \sset \C T^* G$ by specifying its fiber at $x \in G$, which is
\begin{align*}
  \T'_x &\dfn \{ \xi \in \C T^*_x G \st \langle \xi, v \rangle  = 0, \ \forall v \in \VV_x \}.
\end{align*}
For $p, q \in \Z_+$ let $\T'^{p,q}$ be the vector subbundle of $\wedge^{p + q} \C T^* G$ whose fiber at $x \in G$ consists of linear combinations of exterior products $\theta_1 \wedge \cdots \wedge \theta_{p + q}$, where $\theta_j \in \C T^*_x G$ for $1 \leq j \leq p + q$ and at least $p$ of these factors belong to $\T'_x$. Thus we have defined a family of vector bundles over $G$ such that:
\begin{enumerate}
\item $\T'^{p + 1, q - 1}$ is a vector subbundle of $\T'^{p,q}$; and
\item the exterior derivative $\dd$ of $G$ maps sections of $\T'^{p,q}$ to sections of $\T'^{p, q + 1}$.
\end{enumerate}
This allows us to define $\Lambda^{p,q} \dfn \T'^{p,q} / \T'^{p + 1,q - 1}$, which is again a real-analytic vector bundle over $G$ and $\dd$ induces a family of first-order real-analytic differential operators $\dd': \Lambda^{p,q} \rarr \Lambda^{p, q + 1}$ which forms, for each given $p \in \Z_+$, a differential complex of vector bundles and differential operators: in particular, $\dd'$ defines a linear map in the space of smooth sections of $\Lambda^{p,q}$, that is,
\begin{align*}
  \dd': \cinfty(G; \Lambda^{p,q}) \longrightarrow \cinfty(G; \Lambda^{p,q + 1}) 
\end{align*}
thus producing the \emph{cohomology spaces of $\VV$ with smooth coefficients}, here denoted by\footnote{This should not be confused with sheaf cohomology since we are not assuming local exactness of $\dd'$.} $H^{p,q}_{\mathcal{V}}(G; \cinfty(G))$. One can also work in the setting of Schwartz distributions
\begin{align*}
  \dd': \D'(G; \Lambda^{p,q}) \longrightarrow \D'(G; \Lambda^{p,q + 1}) 
\end{align*}
which produces the \emph{cohomology spaces of $\VV$ with distribution coefficients}, now denoted by $H^{p,q}_{\mathcal{V}}(G; \D'(G))$. We can further define other flavors of cohomology by allowing the coefficients to live in other interesting spaces of (generalized) functions; the corresponding notation is implied.

One easily checks, cf.~\cite[Chapter~5]{wallach_haohs}, that each $\Lambda^{p,q}$ is a homogeneous vector bundle and $\dd'$ is a homogeneous differential operator. We shall not make explicit use of these facts: although some of our results are very much related to (and could, in principle, even be derived from) the ones in Wallach's book, we follow a slightly different, simpler path (in particular avoiding the introduction of Representation Theory of Lie groups when it is unnecessary), which we now explain.

We are going to give these objects a more concrete representation. Let $\vv{L}_1, \ldots, \vv{L}_n$ be a basis of $\gr{v}$, the underlying Lie algebra of $\VV$: we regard these as a global frame for $\VV$ (hence a partial frame for $\C T G$) formed by left-invariant complex vector fields. We can adjoin to them left-invariant complex vector fields $\vv{M}_1, \ldots, \vv{M}_m$, in a non-canonical way, such that 
\begin{align}	
  \vv{L}_1, \ldots, \vv{L}_n, \vv{M}_1, \ldots, \vv{M}_m \ \text{is a basis for $\C \gr{g}$} \label{eq:canonical_frame}
\end{align}
(hence $m + n = N$, the dimension of $G$), whose dual basis we denote by $\tau_1, \ldots, \tau_n, \zeta_1, \ldots, \zeta_m \in \C \gr{g}^*$. We regard the latter as left-invariant, complex $1$-forms that form a global frame for $\C T^* G$, dual to $\vv{L}_1, \ldots, \vv{L}_n, \vv{M}_1, \ldots, \vv{M}_m$. This implies, for instance, that if $f \in \cinfty(G)$ then we can write
\begin{align*}
  \dd f &= \sum_{j = 1}^n \vv{L}_j f \ \tau_j + \sum_{k = 1}^m \vv{M}_k f \ \zeta_k.
\end{align*}
More generally, for $p,q \in \Z_+$ we define
\begin{align*}
  \tau_J &\dfn \tau_{j_1} \wedge \cdots \wedge \tau_{j_q}, \ \text{if $J = (j_1, \ldots, j_q)$} \\
  \zeta_I &\dfn \zeta_{i_1} \wedge \cdots \wedge \zeta_{i_p}, \ \text {if $I = (i_1, \ldots, i_p)$}
\end{align*}
so the set $\{ \zeta_I \wedge \tau_J \st \text{$I,J$ ordered, with $|I| + |J| = r$} \}$ is a frame for $\wedge^r \C T^* G$, or a basis for $ \wedge^r \C \gr{g}^*$. Thanks to these facts, any $u \in \cinfty(G; \Lambda^{p,q})$ can be represented as
\begin{align}
  u &= \psum_{|I| = p} \psum_{|J| = q} u_{IJ} \ \zeta_I \wedge \tau_J \label{eq:urepglobal}
\end{align}
where a ``primed'' sum means that we are summing only over \emph{ordered} multi-indices, rendering the representation above unique, and we then have isomorphisms
\begin{align}
  \cinfty(G; \Lambda^{p,q}) &\cong \left\{ \psum_{|I| = p} \psum_{|J| = q} u_{IJ} \ \zeta_I \wedge \tau_J \st u_{IJ} \in \cinfty(G) \right\}. \label{eq:isom_dprime1}
\end{align}
In that sense we regard these as spaces of vector-valued functions on $G$, global coordinates provided by the bases of (classes of) left-invariant forms
\begin{align}
  \{ \zeta_I \wedge \tau_J \st \text{$I,J$ ordered, with $|I| = p$ and $|J| = q$} \} \label{eq:bases_pq}
\end{align}
for $p, q \in \Z_+$. Moreover, we can represent the operator $\dd': \cinfty(G; \Lambda^{p,q}) \rarr \cinfty(G; \Lambda^{p,q + 1})$ under these isomorphisms: if $u \in \cinfty(G; \Lambda^{p,q})$ is as in~\eqref{eq:urepglobal} then
\begin{align*}
  \dd u &= \psum_{|I| = p} \psum_{|J| = q} \dd \left( u_{IJ} \ \zeta_I \wedge \tau_J \right) \\
  &= \psum_{|I| = p} \psum_{|J| = q} \dd u_{IJ} \wedge \zeta_I \wedge \tau_J - \psum_{|I| = p} \psum_{|J| = q} u_{IJ} \ \dd \left(\zeta_I \wedge \tau_J \right) \\
  &= \sum_{j = 1}^n \psum_{|I| = p} \psum_{|J| = q} \vv{L}_j u_{IJ} \ \tau_j \wedge \zeta_I \wedge \tau_J + \sum_{k = 1}^m \psum_{|I| = p} \psum_{|J| = q} \vv{M}_k u_{IJ} \ \zeta_k \wedge \zeta_I \wedge \tau_J - \psum_{|I| = p} \psum_{|J| = q} u_{IJ} \ \dd \left(\zeta_I \wedge \tau_J \right),
\end{align*}
and the second term in the last line above does not ``survive'' in the quotient $\T'^{p,q + 1} / \T'^{p + 1,q}$:
\begin{align*}
  \dd u &= \sum_{j = 1}^n \psum_{|I| = p} \psum_{|J| = q} \vv{L}_j u_{IJ} \ \tau_j \wedge \zeta_I \wedge \tau_J - \psum_{|I| = p} \psum_{|J| = q} u_{IJ} \ \dd \left(\zeta_I \wedge \tau_J \right) \ \text{modulo $\cinfty(G; \T'^{p + 1,q})$}.
\end{align*}
Focusing on the second term above, since $\dd \left(\zeta_I \wedge \tau_J \right)$ is a left-invariant $(p + q + 1)$-form it can be written as a linear combination (with constant coefficients!)~of our basic forms~\eqref{eq:bases_pq}, which amounts to writing its class modulo $\cinfty(G; \T'^{p + 1,q})$ as
\begin{align}
  \dd'( \zeta_I \wedge \tau_J) &= \psum_{|L| = p} \psum_{|K| = q + 1} \alpha _{LK}^{IJ} \ \zeta_L \wedge \tau_K \label{eq:isom_dprime2}
\end{align}
where\footnote{Notice that when $G$ is commutative these are actually zero: since the basic vector fields~\eqref{eq:canonical_frame} are then pairwise commutative their dual forms $\tau_1, \ldots, \tau_n, \zeta_1, \ldots, \zeta_m$ are automatically closed, hence~\eqref{eq:isom_dprime2} vanishes in that case.} $\alpha _{LK}^{IJ} \in \C$ so that
\begin{align}
  \dd' u &= \sum_{j = 1}^n \psum_{|I| = p} \psum_{|J| = q} \vv{L}_j u_{IJ} \ \tau_j \wedge \zeta_I \wedge \tau_J - \psum_{|I| = p} \psum_{|J| = q} u_{IJ} \ \dd' \left(\zeta_I \wedge \tau_J \right). \label{eq:isom_dprime3}
\end{align}
This allows us to treat $\dd'$ w.r.t.~the bases~\eqref{eq:bases_pq} as a matrix of first-order differential operators, each entry of the matrix being a linear combination of left-invariant vector fields and constants.

If one assumes $G$ to be compact then it admits $\ad$-invariant metrics~\cite[Proposition~4.24]{knapp_lgbi}: a left-invariant Riemannian metric on $G$ (and these are in one-to-one correspondence with Euclidean inner products $\langle \cdot, \cdot \rangle$ on $\gr{g}$) is said to be \emph{$\ad$-invariant} if
\begin{align*}
  \langle [\vv{X}, \vv{Y}], \vv{Z} \rangle &= - \langle \vv{Y}, [\vv{X}, \vv{Z}] \rangle, \quad \forall \vv{X}, \vv{Y}, \vv{Z} \in \gr{g}.
\end{align*}
It is well-known that the Laplace-Beltrami operator associated to such a metric commutes with every left-invariant vector field, hence with the entries of (the matrix of) $\dd'$, as we represented it above~\eqref{eq:isom_dprime3}. This is our starting point, and motivates us to study systems of operators that commute with the Laplace-Beltrami operator on a compact, real-analytic Riemannian manifold: it turns out that in studying some properties of such systems one can simply forget about the group structure of $G$.

In this new setup we develop some results concerning regularity and solvability of such operators in several function spaces (Sections~\ref{sec:reg} and~\ref{sec:reg_com}), for which the introduction of the ``formal'' spaces of sequences $\Pi(\Delta)$ and $\Sigma(\Delta)$, as well as the operators naturally acting on them (see Section~\ref{sec:fdos} for definitions) seems to greatly simplify many proofs. Among them, we make explicit the equivalence between a notion of global hypoellipticity-like regularity for these operators (modulo its kernel, thus particularly well-suited for differential complexes; see  Sections~\ref{sec:reg} and~\ref{sec:reg_com}) and topological properties of their range (i.e.~closedness, in many function spaces naturally carrying well-behaved topologies). A similar discussion may be found in~\cite{bp99} for a class of (non-left-invariant) vector fields on a torus in the smooth setup.

These results can then be specialized to left-invariant involutive structures, for instance yielding, roughly speaking:
\begin{quote}
  In a given bidegree $(p,q)$, if $\dd'$ has closed range when acting between smooth spaces, then it also has closed range when acting between Gevrey spaces (of Gevrey order $s \geq 1$). Moreover, if the ``smooth cohomology space'' is finite dimensional then so is the ``Gevrey cohomology space'' (same bidegree), and their dimensions are the same.
\end{quote}
Precise statements may be found in Section~\ref{sec:immediate}, as well as other related results involving spaces of Schwartz distributions, Gevrey ultradistributions and $L^2$ functions (following a suggestion of Andrew Raich which culminated in the inclusion of Section~\ref{sec:l2}).

We close this work discussing a seemly distinct problem, and prove that given a compact and connected Lie group $G$ and $\VV \sset \C T G$  a left-invariant, elliptic and semisimple involutive structure, then all the cohomology spaces $H^{0,q}_\VV(G; \cinfty(G))$ are left-invariant (see Section~\ref{sec:licoho} for definitions). This question is related to a well-known result due to Chevalley and Eilenberg~\cite{ce48} about the left-invariance of de Rham cohomology classes on $G$, and is more extensively discussed in~\cite{jahnke_thesis} for other classes of elliptic structures.

We were not able, however, to prove the existence of such structures (simultaneously elliptic and semisimple). This is only excused by our sole intention to illustrate further potential applications of the methods developed in the preceding sections.

\subsection*{Acknowledgments} I wish to thank Paulo D.~Cordaro and Andrew Raich for discussing parts of this work and their very useful inputs. Also especially Max R.~Jahnke and Luis F.~Ragognette for their active participation in the earlier stages of this work, including helping to set up the original questions and proposing the framework that led to it, as well as many helpful suggestions throughout its development. \nocite{rodrigues_thesis}

\section{A class of formal operators on compact manifolds} \label{sec:fdos}
Let $\Omega$ be a compact, connected, real-analytic manifold. For simplicity we further require it to be orientable and in fact oriented. We endow $\Omega$ with a real-analytic Riemannian metric, whose associated volume form induces a Radon measure on $\Omega$: the $L^2$ norms below are taken with respect to this measure.

We denote by $\Delta$ the Laplace-Beltrami operator associated with our metric, which is a second-order, positive semidefinite, elliptic LPDO with real-analytic coefficients on $\Omega$. Its spectrum $\sigma(\Delta) \sset \R_+$ is countable, and we denote by $E_\lambda \dfn \ker (\Delta - \lambda I)$ the eigenspace associated with $\lambda \in \sigma(\Delta)$: since $\Delta$ is elliptic, this is a finite dimensional complex vector space of real-analytic functions. Very important for our purposes is the following consequence of Weyl's asymptotic formula~\cite[p.~155]{chavel_eigenvalues}:
\begin{align}
  \sum_{\sigma(\Delta) \setminus 0} (\dim E_\lambda) \lambda^{-2N} &< \infty \label{eq:weyl}
\end{align}
where $N \dfn \dim \Omega$. We also point out that $E_0$ is precisely the space of constant functions since $\Omega$ is connected. By considering each $E_\lambda$ endowed with the $L^2$ inner product, spectral theory tells us that
\begin{align}
  L^2(\Omega) &\cong \hsum{\sigma(\Delta)} E_\lambda \label{eq:dec_delta}
\end{align}
as Hilbert spaces: the right-hand side of~\eqref{eq:dec_delta} (which we will denote by $L^2(\Delta)$ when we want to formally distinguish it from $L^2(\Omega)$) is the Hilbert space direct sum of the family $\{ E_\lambda \st \lambda \in \sigma(\Delta) \}$ i.e.~the space of all sequences $a \dfn (a(\lambda))_{\lambda \in \sigma(\Delta)}$ with $a(\lambda) \in E_\lambda$ and such that $\{ \| a(\lambda) \|_{L^2(\Omega)}^2 \}_{\lambda \in \sigma(\Delta)}$ is summable, endowed with the Hilbert space norm
\begin{align*}
  \| a \|_{L^2(\Delta)} &\dfn \left( \sum_{\sigma(\Delta)} \| a(\lambda) \|_{L^2(\Omega)}^2 \right)^{\frac{1}{2}}.
\end{align*}

As we will see, moving completely to spaces of sequences makes many proofs a lot easier. We introduce the spaces
\begin{align*}
  \Pi(\Delta) \dfn \prod_{\sigma(\Delta)} E_\lambda, & \quad \Sigma(\Delta) \dfn \bigoplus_{\sigma(\Delta)} E_\lambda
\end{align*}
which we interpret as the space of all sequences $a \dfn (a(\lambda))_{\lambda \in \sigma(\Delta)}$ with $a(\lambda) \in E_\lambda$ (without constraints), and its subspace of all eventually null such sequences, respectively. They carry natural locally convex topologies (product and direct sum topologies, respectively), and in this sense we can express $\Pi(\Delta)$ (resp.~$\Sigma(\Delta)$) as the locally convex projective (resp.~injective) limit of the family
\begin{align*}
  \left\{ \bigoplus_F E_\lambda \st \text{$F \sset \sigma(\Delta)$ is finite} \right\}
\end{align*}
in an appropriate sense. In particular, $\Pi(\Delta)$ and $\Sigma(\Delta)$ are an FS and a DFS space, respectively: for definitions and the main results about these kind of spaces (which we will eventually need) see~\cite{kom67}. Moreover, the bilinear pairing
\begin{align}
  (u, v) \in \Pi(\Delta) \times \Sigma(\Delta) &\longmapsto \sum_{\sigma(\Delta)} \langle u(\lambda), \overline{v(\lambda)} \rangle_{L^2(\Omega)} \in \C \label{eq:fund_pairing_pisigma}
\end{align}
turns these spaces into the dual of one another.

In this setting, if for each $\lambda \in \sigma(\Delta)$ we denote by $\mathcal{F}_\lambda: L^2(\Omega) \rarr E_\lambda$ the corresponding orthogonal projection then~\eqref{eq:dec_delta} means that the linear map
\begin{align*}
  \TR{\mathcal{F}}{f}{L^2(\Omega)}{ \left( \mathcal{F}_\lambda(f) \right)_{\lambda \in \sigma(\Delta)}}{\Pi(\Delta)}
\end{align*}
maps $L^2(\Omega)$ isometrically onto $L^2(\Delta)$, with inverse given by the formula
\begin{align}
  f &= \sum_{\sigma(\Delta)} \mathcal{F}_\lambda(f) \label{eq:abst_fourier_exp} 
\end{align}
where convergence takes place in $L^2(\Omega)$ for any $f \in L^2(\Omega)$.

\begin{Rem} \label{rem:rep_theory} On time, we briefly recall how this is related to the group theoretic Fourier Analysis in a compact Lie group $G$. Let $(\xi, V_\xi)$ be an irreducible unitary (complex) representation of $G$, hence $V_\xi$ is finite dimensional and we denote $d_\xi \dfn \dim_\C V_\xi$. We select $\{ v_1, \ldots, v_{d_\xi}\}$ an orthonormal basis for $V_\xi$ and define, for $j,k \in \{1, \ldots, d_\xi\}$:
\begin{align*}
  \TR{\xi_{jk}}{x}{G}{\langle \xi(x)v_j, v_k \rangle_\xi}{\C}
\end{align*}
the so-called \emph{matrix elements} of $(\xi, V_\xi)$. These are smooth functions on $G$ and we denote by $\mathcal{M}_\xi$ the subspace of $L^2(G)$ spanned by $\{ \xi_{jk} \st 1 \leq j,k \leq d_\xi \}$. It is easily checked that the matrix elements are pairwise orthogonal and $\| \xi_{jk} \|_{L^2(G)}^2 = 1/d_\xi$, hence $\{ \sqrt{d_\xi} \xi_{jk} \st 1 \leq j,k \leq d_\xi \}$ is an orthonormal basis for $\mathcal{M}_\xi$. In particular, $\dim_\C \mathcal{M}_\xi = d_\xi^2$. Yet, the definition of $\mathcal{M}_\xi$ does not depend on the choice of $v_1, \ldots, v_{d_\xi}$.

The Peter-Weyl Theorem states that if $(\xi, V_\xi)$ and $(\eta, V_\eta)$ are two non-equivalent irreducible unitary representations of $G$ then $\mathcal{M}_\xi \bot \mathcal{M}_\eta$ in $L^2(G)$ and, moreover,
\begin{align*}
  L^2(G) &\cong \hsum{[\xi] \in \widehat{G}} \mathcal{M}_\xi
\end{align*}
where $\widehat{G}$ is the set of all (equivalence classes of) irreducible unitary representations\footnote{We are actually summing over a set of representations of $G$ containing exactly one representative of each class in $\widehat{G}$.} of $G$. Actually, every $f \in L^2(G)$ can be written as
\begin{align}
  f &= \sum_{[\xi] \in \widehat{G}} d_\xi \sum_{j,k = 1}^{d_\xi} \langle f, \xi_{jk} \rangle_{L^2(G)} \xi_{jk} \label{eq:fourier_pw}
\end{align}
with convergence in $L^2(G)$.

Let $\Delta$ denote the Laplace-Beltrami operator w.r.t.~some $\ad$-invariant metric on $G$. Given an irreducible unitary representation $(\xi, V_\xi)$, it turns out that there exists $\lambda_\xi \geq 0$ (depending only on the equivalence class $[\xi] \in \widehat{G}$) such that
\begin{align*}
  \Delta f &= \lambda_\xi f, \quad \forall f \in \mathcal{M}_\xi
\end{align*}
i.e.~$\lambda_\xi \in \sigma(\Delta)$. Conversely, thanks to the orthogonal decomposition~\eqref{eq:fourier_pw} we have that if $\lambda \in \sigma(\Delta)$ then $\lambda = \lambda_\xi$ for some $[\xi] \in \widehat{G}$ (not necessarily unique). One then immediately realizes that~\eqref{eq:fourier_pw} is actually a refined version of~\eqref{eq:abst_fourier_exp}: for each $\lambda \in \sigma(\Delta)$ we have an orthogonal decomposition
\begin{align*}
  E_\lambda &= \bigoplus_{\substack{[\xi] \in \widehat{G} \\ \lambda_\xi = \lambda}} \mathcal{M}_\xi
\end{align*}
and
\begin{align*}
  \mathcal{F}_\lambda(f) &= \sum_{\substack{[\xi] \in \widehat{G} \\ \lambda_\xi = \lambda}} d_\xi \sum_{j,k = 1}^{d_\xi} \langle f, \xi_{jk} \rangle_{L^2(G)} \xi_{jk}, \quad \forall f \in L^2(G).
\end{align*}

It turns out, however, that in the present work this refinement will not be needed, so we may continue with our compact manifold $\Omega$ as before.
\end{Rem}

A continuous linear map $P: \Pi(\Delta) \rarr \Pi(\Delta)$ that commutes with $\Delta$ will be called a \emph{$\Delta$-invariant formal differential operator}, or $\Delta$FDO for short. To be precise, since $\Delta$ maps each eigenspace $E_\lambda$ into itself, we can regard it as a continuous linear map
\begin{align*}
  \TR{\Delta}{a}{\Pi(\Delta)}{\left( \Delta a(\lambda) \right)_{\lambda \in \sigma(\Delta)}}{\Pi(\Delta)}
\end{align*}
and in that sense we impose the relation $[P, \Delta] = 0$, meaning that $P \circ \Delta = \Delta \circ P$ as linear endomorphisms of $\Pi(\Delta)$. The space of all $\Delta$FDOs is a subalgebra of $\End(\Pi(\Delta))$.

For each $\lambda \in \sigma(\Delta)$, we identify $E_{\lambda}$ with its image under $\mathcal{F}$, the space of all $a \in \Pi(\Delta)$ such that $a(\eta) = 0$ for $\eta \neq \lambda$. Under this identification we have, for $\phi \in E_{\lambda}$,
\begin{align*}
  \Delta(P\phi) = P(\Delta \phi) = P(\lambda \phi) = \lambda (P\phi)
\end{align*}
that is $P(E_\lambda) \sset E_\lambda$: we denote the corresponding linear map by $\widehat{P}(\lambda) \in \End(E_\lambda)$, which, to be precise, is the unique one that renders the diagram
\begin{align*} 
  \xymatrix{
    \Pi(\Delta) \ar[r]^{P} & \Pi(\Delta)  \\
    E_\lambda \ar[u]^{\mathcal{F}} \ar[r]_{\widehat{P}(\lambda)} &  E_\lambda \ar[u]_{\mathcal{F}}
  } 
\end{align*}
into a commutative one. This can be easily extended to \emph{finite} sums of eigenvectors: let $\lambda_1, \ldots, \lambda_k \in \sigma(\Delta)$ be distinct and let $\phi_j \in E_{\lambda_j}$ for $j \in \{1, \ldots, k\}$; for $a \dfn \phi_1 + \cdots + \phi_k$ we have
\begin{align}
  (Pa)(\lambda) &= \widehat{P}(\lambda) a(\lambda), \quad \forall \lambda \in \sigma(\Delta). \label{eq:nice_continuity_dfdo}
\end{align}
\begin{Prop} \label{prop:nice_continuity_dfdo} Property~\eqref{eq:nice_continuity_dfdo} holds for every $a \in \Pi(\Delta)$.
  \begin{proof} Let $\{ F_\nu \}_{\nu \in \N}$ be a family of finite subsets of $\sigma(\Delta)$ such that $F_\nu \nearrow \sigma(\Delta)$, and for each $\nu \in \N$ let $a_\nu \in \Pi(\Delta)$ be defined by the truncation
    \begin{align*}
      a_\nu(\lambda) &\dfn
      \begin{cases}
        a(\lambda), &\text{if $\lambda \in F_\nu$} \\
        0, &\text{otherwise}
      \end{cases}
    \end{align*}
    which as a finite sum of eigenvectors of $\Delta$ satisfies~\eqref{eq:nice_continuity_dfdo} i.e.~$(Pa_\nu)(\lambda) = \widehat{P}(\lambda) a_\nu(\lambda)$ for every $\lambda \in \sigma(\Delta)$. Also, since $a_\nu(\lambda) \to a(\lambda)$ for every $\lambda \in \sigma(\Delta)$ we have $a_\nu \to a$ in $\Pi(\Delta)$, and continuity of $P$ finally kicks in: we have $Pa_\nu \to Pa$ in $\Pi(\Delta)$, which again means that $(Pa_\nu)(\lambda) \to (Pa)(\lambda)$ for every $\lambda \in \sigma(\Delta)$. The conclusion follows at once.
  \end{proof}
\end{Prop}
As an easy corollary, we notice that every $\Delta$FDO maps $\Sigma(\Delta)$ into itself.

\begin{Rem} Given a $\Delta$FDO $P$ we can assemble a map
\begin{align*}
  \TR{\widehat{P}}{\lambda}{\sigma(\Delta)}{\widehat{P}(\lambda)}{\prod_{\sigma(\Delta)} \End(E_\lambda)}
\end{align*}
which we may interpret as a kind of symbol for $P$. Actually, we might call a \emph{symbol} any map
\begin{align*}
  \Phi: \sigma(\Delta) \longrightarrow \prod_{\sigma(\Delta)} \End(E_\lambda)
\end{align*}
such that $\Phi(\lambda) \in \End(E_\lambda)$ for every $\lambda \in \sigma(\Delta)$: they form an algebra under pointwise composition, and the map that associates to each $\Delta$FDO its symbol is clearly an isomorphism of algebras. We may regard its inverse
\begin{align*}
  \Phi &\longmapsto \left\{
  \TR{\Phi(\Delta)}{a}{\Pi(\Delta)}{ \left( \Phi(\lambda)a(\lambda) \right)_{\lambda \in \sigma(\Delta)} }{\Pi(\Delta)}
  \right.
\end{align*}
as a kind of quantization map (it is easy to show that the quantization of any symbol is continuous in $\Pi(\Delta)$, hence a true $\Delta$FDO). This point-of-view provides an alternative proof for Proposition~\ref{prop:nice_continuity_dfdo}: both $P$ and the quantization of $\widehat{P}$ are continuous endomorphisms of $\Pi(\Delta)$ that, according to~\eqref{eq:nice_continuity_dfdo}, coincide on $\Sigma(\Delta)$ -- a dense subspace of $\Pi(\Delta)$ --, so they must be equal everywhere. However, the conclusion of Proposition~\ref{prop:nice_continuity_dfdo} is trivial for any $\Delta$FDO defined by quantization of a symbol!

It should also be pointed out that all of this works even if $P$ were defined in a much smaller space to start with. Indeed, suppose that $P: \Sigma(\Delta) \rarr \Pi(\Delta)$ is any linear map that commutes with $\Delta$ (which, by the way, also acts as a linear endomorphism of $\Sigma(\Delta)$ in an obvious way) then again $P(E_\lambda) \sset E_\lambda$ for every $\lambda \in \sigma(\Delta)$: this implies that $P$ is automatically continuous and also allows us to define its symbol $\widehat{P}$ in the same manner as we did before. We can then quantize $\widehat{P}$ by
\begin{align*}
  \TR{\tilde{P}}{a}{\Pi(\Delta)}{\left( \widehat{P}(\lambda) a(\lambda) \right)_{\lambda \in \sigma(\Delta)}}{\Pi(\Delta)}
\end{align*}
which is a very natural linear extension of $P$ to $\Pi(\Delta)$ -- actually, the unique continuous one: notice that $\tilde{P}$ is a $\Delta$FDO and its symbol is precisely $\widehat{P}$. The foremost example we have in mind is of course a linear differential operator $P: \cinfty(\Omega) \rarr \cinfty(\Omega)$ that commutes with $\Delta$ e.g.~a left-invariant differential operator (when $\Omega$ is a Lie group).
\end{Rem}

More generally, we will deal with systems of such objects. A matrix $P = (P_{ij})_{n \times m}$ of $\Delta$FDOs naturally defines a continuous linear action $P: \Pi(\Delta)^m \rarr \Pi(\Delta)^n$: if $a \dfn (a_1, \ldots, a_m) \in \Pi(\Delta)^m$ then $Pa \in \Pi(\Delta)^n$ is defined by
\begin{align*}
  (Pa)_i &\dfn \sum_{j = 1}^m P_{ij} a_j, \quad \forall i \in \{1, \ldots, n\}.
\end{align*}
It is clear that for each $\lambda \in \sigma(\Delta)$ we have $P(E_\lambda^m) \sset E_\lambda^n$ and the induced linear map is precisely the one defined by the matrix $\widehat{P}(\lambda) \dfn ( \widehat{P}_{ij}(\lambda) )_{n \times m}$ which we regard as a linear map $\widehat{P}(\lambda): E_\lambda^m \rarr E_\lambda^n$. For $a \in \Pi(\Delta)^m$ we also define $a(\lambda) \dfn (a_1(\lambda), \ldots, a_m(\lambda)) \in E_\lambda^m$, so in particular $(Pa)(\lambda) = \widehat{P}(\lambda)a(\lambda)$ for every $\lambda \in \sigma(\Delta)$, just as in the scalar case (Proposition~\ref{prop:nice_continuity_dfdo}).

Now let $P = (P_{ij})_{n \times m}$ and $Q = (Q_{jk})_{m \times r}$ be two matrices of $\Delta$FDOs forming a differential complex
\begin{align}
  \xymatrix{
    \Pi(\Delta)^r \ar[r]^Q & \Pi(\Delta)^m \ar[r]^P & \Pi(\Delta)^n 
  } \label{eq:complexPQ}
\end{align}
i.e.~$P \circ Q = 0$. By looking at their symbols, it is clear that for each $\lambda \in \sigma(\Delta)$ we have a new differential complex, now involving only finite dimensional vector spaces:
\begin{align*}
  \xymatrix{
    E_\lambda^r \ar[r]^{\widehat{Q}(\lambda)} & E_\lambda^m \ar[r]^{\widehat{P}(\lambda)} & E_\lambda^n. 
  } 
\end{align*}
We are interested in studying several flavors of cohomology for the complex~\eqref{eq:complexPQ}. For instance, assume that there exists a linear subspace $\mathscr{V} \sset \Pi(\Delta)$ such that $Q$ maps $\mathscr{V}^r$ into $\mathscr{V}^m$, which is then mapped by $P$ into $\mathscr{V}^n$. We have a new differential complex
\begin{align}
  \xymatrix{
    \mathscr{V}^r \ar[r]^Q & \mathscr{V}^m \ar[r]^P & \mathscr{V}^n 
  } \label{eq:complexPQ_V}
\end{align}
whose cohomology we denote by
\begin{align*}
  \mathcal{H}_{P,Q} (\mathscr{V}) &\dfn \frac{ \ker \{ P: \mathscr{V}^m \rarr \mathscr{V}^n \} }{ \ran \{ Q: \mathscr{V}^r \rarr \mathscr{V}^m \} }.
\end{align*}

\begin{Lem} \label{lem:algebraic_isoms_delta} We have natural isomorphisms of vector spaces
  \begin{align*}
    \mathcal{H}_{P,Q} (\Sigma(\Delta)) \cong \bigoplus_{\sigma(\Delta)} \frac{\ker \widehat{P}(\lambda)}{\ran \widehat{Q}(\lambda)}, & \quad \mathcal{H}_{P,Q} (\Pi(\Delta)) \cong \prod_{\sigma(\Delta)} \frac{\ker \widehat{P}(\lambda)}{\ran \widehat{Q}(\lambda)}.
  \end{align*}
  \begin{proof} We prove the first isomorphism, for the second one is very similar (and actually a bit easier). For any $a \in \Sigma(\Delta)^m$ belonging to $\ker P$ we have
    \begin{align*}
      \widehat{P}(\lambda)a(\lambda) = (Pa)(\lambda) = 0, \quad \forall \lambda \in \sigma(\Delta),
    \end{align*}
    i.e.~$a(\lambda) \in \ker \widehat{P}(\lambda)$ for every $\lambda \in \sigma(\Delta)$. Because any such $a$ will, by definition of direct sum, have at most finitely many non-zero components $a(\lambda)$, we can naturally associate to it an element
    \begin{align*}
      A(a) &\in \bigoplus_{\sigma(\Delta)} \frac{\ker \widehat{P}(\lambda)}{\ran \widehat{Q}(\lambda)}
    \end{align*}
    which depends on the class of $a$ in $\mathcal{H}_{P,Q} (\Sigma(\Delta))$ only. The induced linear map
    \begin{align*}
      A: \mathcal{H}_{P,Q} (\Sigma(\Delta)) \longrightarrow \bigoplus_{\sigma(\Delta)} \frac{\ker \widehat{P}(\lambda)}{\ran \widehat{Q}(\lambda)}
    \end{align*}
    is clearly onto, and also injective. Indeed, if $A(a) = 0$ then for every $\lambda \in \sigma(\Delta)$ we have $a(\lambda) \in \ran \widehat{Q}(\lambda)$ -- and most of these are zero --, hence there exists $u(\lambda) \in E_\lambda^r$ such that $\widehat{Q}(\lambda)u(\lambda) = a(\lambda)$: this defines an element $u \in \Sigma(\Delta)^r$ such that $Qu = a$, so the class of $a$ in $\mathcal{H}_{P,Q} (\Sigma(\Delta))$ is zero.
  \end{proof}
\end{Lem}

\begin{Rem} Clearly:
  \begin{align*}
    \mathcal{H}_{P,Q}(E_\lambda) &= \frac{\ker \widehat{P}(\lambda)}{\ran \widehat{Q}(\lambda)}, \quad \forall \lambda \in \sigma(\Delta).
  \end{align*}
\end{Rem}

The following consequence of the previous lemma is immediate.
\begin{Cor} \label{cor:cond_fin_dim} All the conditions below are equivalent.
  \begin{enumerate}
  \item $\mathcal{H}_{P,Q} (\Sigma(\Delta))$ is finite dimensional.
  \item $\mathcal{H}_{P,Q} (\Pi(\Delta))$ is finite dimensional.
  \item $\ker \widehat{P}(\lambda) = \ran \widehat{Q}(\lambda)$ for all but finitely many $\lambda \in \sigma(\Delta)$.
  \end{enumerate}
  In that case
  \begin{align*}
    \dim \mathcal{H}_{P,Q} (\Pi(\Delta)) = \dim \mathcal{H}_{P,Q} (\Sigma(\Delta)) = \sum_{\sigma(\Delta)} \dim \left( \frac{\ker \widehat{P}(\lambda)}{\ran \widehat{Q}(\lambda)} \right).
  \end{align*}
\end{Cor}

The inclusion map $\Sigma(\Delta)^m \hookrightarrow \Pi(\Delta)^m$ induces a linear map
\begin{align}
  \mathcal{H}_{P,Q}(\Sigma(\Delta)) \longrightarrow \mathcal{H}_{P,Q}(\Pi(\Delta)) \label{eq:inj_map_cohomologies} 
\end{align}
that is \emph{always} injective. Indeed, let $a \in \Sigma(\Delta)^m$ represent a class in $\mathcal{H}_{P,Q}(\Sigma(\Delta))$ such that, as an element of $\Pi(\Delta)^m$, belongs to $\ran \{ Q: \Pi(\Delta)^r \rarr \Pi(\Delta)^m \}$. This means that there exists $v \in \Pi(\Delta)^r$ -- perhaps with infinitely many non-zero entries -- such that $Qv = a$, that is, $\widehat{Q}(\lambda)v(\lambda) = a(\lambda)$ for all $\lambda \in \sigma(\Delta)$. Define $u \in \Pi(\Delta)^r$ by
\begin{align*}
  u(\lambda) &\dfn
  \begin{cases}
    v(\lambda), &\text{if $a(\lambda) \neq 0$}, \\
    0, &\text{if $a(\lambda) = 0$},
  \end{cases}
\end{align*}
hence $u$ actually belongs to $\Sigma(\Delta)^r$ since $a \in \Sigma(\Delta)^m$ by hypothesis. Obviously $\widehat{Q}(\lambda) u(\lambda) = a(\lambda)$ for every $\lambda \in \sigma(\Delta)$ i.e.~$Qu = a$, proving that the class of $a$ in $\mathcal{H}_{P,Q}(\Sigma(\Delta))$ is zero. Corollary~\ref{cor:cond_fin_dim} tells us that if either of these cohomology spaces is finite dimensional then~\eqref{eq:inj_map_cohomologies} is an isomorphism.

Now, let us go back to the situation in~\eqref{eq:complexPQ_V}, but assume further that $\Sigma(\Delta) \sset \mathscr{V}$. The inclusion maps $\Sigma(\Delta)^m \hookrightarrow \mathscr{V}^m \hookrightarrow \Pi(\Delta)^m$ induce linear maps
\begin{align}
  \mathcal{H}_{P,Q}(\Sigma(\Delta)) \longrightarrow \mathcal{H}_{P,Q}(\mathscr{V}) \longrightarrow \mathcal{H}_{P,Q}(\Pi(\Delta)) \label{eq:natural_maps_cohomology}
\end{align}
the first of them clearly injective (since so is their composition). This implies at once:
\begin{Prop} \label{prop:findimV} If $\mathcal{H}_{P,Q} (\mathscr{V})$ is finite dimensional then so is $\mathcal{H}_{P,Q} (\Sigma(\Delta))$ and, in that case,
  \begin{align}
    \dim \mathcal{H}_{P,Q} (\mathscr{V}) &\geq \sum_{\sigma(\Delta)} \dim \left( \frac{\ker \widehat{P}(\lambda)}{\ran \widehat{Q}(\lambda)} \right). \label{eq:basic_ineq_dim}
  \end{align}
\end{Prop}

In the next section, we will construct many spaces $\mathscr{V}$ to which we are interested in applying these results. We shall also pursue conditions (necessary and sufficient) when the inequality~\eqref{eq:basic_ineq_dim} is actually an equality: we close this section with two results that, although very simple and general, will be invaluable in that direction.
\begin{Lem} \label{lem:reg_impl_inject_cohom} Assume that $Q$ satisfies the following regularity condition w.r.t.~$\mathscr{V}$:
  \begin{align}
    \forall u \in \Pi(\Delta)^r, \ Qu \in \mathscr{V}^m &\Longrightarrow \text{$\exists v \in \ker Q$ such that $u - v \in \mathscr{V}^r$}. \label{eq:abstract_reg}
  \end{align}
  Then the natural map $\mathcal{H}_{P,Q}(\mathscr{V}) \rightarrow \mathcal{H}_{P,Q}(\Pi(\Delta))$ is injective.
  \begin{proof} Let $a \in \mathscr{V}^m$ represent a class in $\mathcal{H}_{P,Q}(\mathscr{V})$ whose corresponding class in $\mathcal{H}_{P,Q}(\Pi(\Delta))$ is zero: there exists $u \in \Pi(\Delta)^r$ such that $Qu = a$. By hypothesis, there exists $v \in \ker Q$ such that $u - v \in \mathscr{V}^r$, and clearly $Q(u - v) = Qu = a$, thus proving that the class of $a$ in $\mathcal{H}_{P,Q}(\mathscr{V})$ is zero.
  \end{proof}
\end{Lem}

\begin{Cor} \label{cor:isoms_abstract} Suppose that $\dim \mathcal{H}_{P,Q}(\mathscr{V}) < \infty$ and $Q$ satisfies~\eqref{eq:abstract_reg}. Then the natural maps~\eqref{eq:natural_maps_cohomology} are all isomorphisms; in particular
  \begin{align}
    \dim \mathcal{H}_{P,Q} (\mathscr{V}) &= \sum_{\sigma(\Delta)} \dim \left( \frac{\ker \widehat{P}(\lambda)}{\ran \widehat{Q}(\lambda)} \right). \label{eq:basic_eq_dim}
  \end{align}
  \begin{proof} By Lemma~\ref{lem:reg_impl_inject_cohom} and our previous digression all the aforementioned morphisms are injective, so finite dimensionality of $\mathcal{H}_{P,Q}(\mathscr{V})$ implies the same property for $\mathcal{H}_{P,Q}(\Sigma(\Delta))$. This, in turn, implies that $\mathcal{H}_{P,Q}(\Pi(\Delta))$ is finite dimensional by Corollary~\ref{cor:cond_fin_dim}, which also yields our final conclusion.
  \end{proof}
\end{Cor}

\section{Prelude: $L^2$ theory of $\Delta$FDOs} \label{sec:l2}

In what follows we will introduce some abstract classes of subspaces of $\Pi(\Delta)$ which will help us to describe the image under $\mathcal{F}$ of several function spaces (subspaces of $L^2(\Omega)$ as well as their generalized counterparts) in a useful way to study regularity (by which we mean hypoellipticity-like properties) of $\Delta$FDOs. In the next section, we introduce this framework that allows us to treat simultaneously the space of smooth function and the spaces of Gevrey functions (including real-analytic functions, and perhaps more general Denjoy-Carleman classes with minor modifications).

In the present section, however, we start with a small digression on how these operators act on $L^2(\Delta)$, which turns out to be surprisingly simple, gives a general feeling on what we will do next and provides some motivation for the forthcoming steps. From now on we state and prove our results mostly for single scalar operators, the adaptation to systems being immediate.

Let $P$ be a $\Delta$FDO, which we regard as a densely defined operator $P: D_P \sset L^2(\Delta) \rarr L^2(\Delta)$, where
\begin{align*}
  D_P &\dfn \{ u \in L^2(\Delta) \st Pu \in L^2(\Delta) \}
\end{align*}
(clearly $\Sigma(\Delta) \sset D_P$), which is closed. Indeed, let $\{ u_\nu \}_{\nu \in \N} \sset D_P$ and $u, f \in L^2(\Delta)$ be such that $u_\nu \to u$ and $Pu_\nu \to f$ in $L^2(\Delta)$. This implies that $u_\nu(\lambda) \to u(\lambda)$ and $\widehat{P}(\lambda)u_\nu(\lambda) \to f(\lambda)$ for all $\lambda \in \sigma(\Delta)$, hence $\widehat{P}(\lambda)u(\lambda) = f(\lambda)$ for all $\lambda \in \sigma(\Delta)$ i.e.~$Pu = f$, which in turn ensures that $u \in D_P$.

\begin{Prop} \label{prop:closedL2} The operator $P: D_P \sset L^2(\Delta) \rarr L^2(\Delta)$ has a closed range if and only if there exists $C > 0$ such that
  \begin{align}
    \| \widehat{P}(\lambda) \phi \|_{L^2(\Omega)} &\geq C, \quad \forall \phi \in \ker \widehat{P}(\lambda)^\bot, \ \| \phi \|_{L^2(\Omega)} = 1, \ \forall \lambda \in \sigma(\Delta). \label{eq:closedL2}
  \end{align}
  \begin{proof} Assume~\eqref{eq:closedL2} and let $\{ u_\nu \}_{\nu \in \N} \sset D_P$ and $f \in L^2(\Delta)$ be such that $Pu_\nu \to f$ in $L^2(\Delta)$. Since then $\widehat{P}(\lambda)u_\nu(\lambda) \to f(\lambda)$, there exists $u(\lambda) \in E_\lambda$ such that $\widehat{P}(\lambda)u(\lambda) = f(\lambda)$ since $E_\lambda$ is finite dimensional, which also allows us to choose $u(\lambda) \in \ker \widehat{P}(\lambda)^\bot$. This holds for every $\lambda \in \sigma(\Delta)$ and defines $u \dfn (u(\lambda))_{\lambda \in \sigma(\Delta)} \in \Pi(\Delta)$ solving $Pu = f$. But from this and~\eqref{eq:closedL2} one easily derives
    \begin{align*}
      \| f(\lambda) \|_{L^2(\Omega)} = \| \widehat{P}(\lambda) u(\lambda) \|_{L^2(\Omega)} \geq C \| u(\lambda) \|_{L^2(\Omega)}, \quad \forall \lambda \in \sigma(\Delta),
    \end{align*}
    hence $u$ belongs to $L^2(\Delta)$, and then actually to $D_P$.

    Now we assume that~\eqref{eq:closedL2} does not hold: we can find an increasing sequence $\{ \lambda_\nu \}_{\nu \in \N} \sset \sigma(\Delta)$ and $\phi_\nu \in \ker \widehat{P}(\lambda_\nu)^\bot$ with $\| \phi_\nu \|_{L^2(\Omega)} = 1$ for each $\nu \in \N$ satisfying
    \begin{align}
      \| \widehat{P}(\lambda_\nu) \phi_\nu \|_{L^2(\Omega)} &\leq \nu^{-1}, \quad \forall \nu \in \N. \label{eq:notclosedL2}
    \end{align}
    We define $u \in \Pi(\Delta)$ by
    \begin{align*}
      u(\lambda) &\dfn
      \begin{cases}
        \phi_\nu, & \text{if $\lambda = \lambda_\nu$ for some $\nu$} \\
        0, & \text{otherwise},
      \end{cases}
    \end{align*}
    which clearly does not belong to $L^2(\Delta)$, while $f \dfn Pu$ does thanks to~\eqref{eq:notclosedL2}. Now we approximate $u$ by truncations as we did in Proposition~\ref{prop:nice_continuity_dfdo}: we take $\{ F_\nu \}_{\nu \in \N}$ a family of finite subsets of $\sigma(\Delta)$ such that $F_\nu \nearrow \sigma(\Delta)$ and let $u_\nu \in \Sigma(\Delta)$ be defined by $u_\nu(\lambda) \dfn u(\lambda)$ if $\lambda \in F_\nu$ and $u_\nu(\lambda) \dfn 0$ otherwise. But then
    \begin{align*}
      \| Pu_\nu - f \|_{L^2(\Delta)}^2 = \sum_{\sigma(\Delta)} \| \widehat{P}(\lambda)u_\nu(\lambda) - f(\lambda) \|_{L^2(\Omega)}^2 = \sum_{\sigma(\Delta) \setminus F_\nu} \| f(\lambda) \|_{L^2(\Omega)}^2 \to 0 \ \text{as $\nu \to \infty$}
    \end{align*}
    i.e.~$Pu_\nu \to f$ in $L^2(\Delta)$. However, we claim that $Pv = f$ for no $v \in L^2(\Delta)$. Indeed, if $v \in \Pi(\Delta)$ satisfies $Pv = f = Pu$ then $\widehat{P}(\lambda)v(\lambda) = \widehat{P}(\lambda)u(\lambda)$ for all $\lambda \in \sigma(\Delta)$ so $v(\lambda) - u(\lambda) \in \ker \widehat{P}(\lambda)$ for all $\lambda \in \sigma(\Delta)$. Since $u(\lambda) \in \ker \widehat{P}(\lambda)^\bot$ we have
    \begin{align*}
      \| v(\lambda) \|_{L^2(\Omega)}^2 = \| v(\lambda) - u(\lambda) \|_{L^2(\Omega)}^2 + \| u(\lambda) \|_{L^2(\Omega)}^2 \geq \| u(\lambda) \|_{L^2(\Omega)}^2
    \end{align*}
    and since $u$ does not belong to $L^2(\Delta)$ then neither does $v$.
  \end{proof}
\end{Prop}

\begin{Exa} If $\{ \lambda \in \sigma(\Delta) \st \widehat{P}(\lambda) \neq 0 \}$ is finite then certainly there exists $C > 0$ such that~\eqref{eq:closedL2} holds.
\end{Exa}
\section{Spaces defined by weight functions}

The ideas developed in this and the next sections are very much related to the ones in~\cite{seeley69} and~\cite{fm95} (see also the more recent works~\cite{dr14},~\cite{dr16}). Let $\omega: \sigma(\Delta) \rarr \R_+$ be an arbitrary function. For each $t \in \R$ we define
\begin{align*}
  \D_{\omega, t} (\Delta) \dfn \left\{ a \in \Pi(\Delta) \st \left( e^{t \omega(\lambda)} a (\lambda) \right)_{\lambda \in \sigma(\Delta)} \in L^2(\Delta) \right\} = e^{-t \omega(\Delta)} L^2(\Delta).
\end{align*}
It is a Hilbert space with norm
\begin{align*}
  \| a \|_{\D_{\omega, t} (\Delta)} \dfn \left( \sum_{\sigma(\Delta)} e^{2t \omega(\lambda)} \|a (\lambda)\|_{L^2(\Omega)}^2 \right)^{\frac{1}{2}} = \| e^{t \omega(\Delta)} a \|_{L^2(\Delta)}.
\end{align*}
We shall always assume that $\omega$ is also increasing and unbounded. A standard argument then ensures that the inclusion map $\D_{\omega, t_+} (\Delta) \hookrightarrow \D_{\omega, t} (\Delta)$ is compact whenever $t_+ > t$, for which it is useful to notice first the following general fact: if a sequence $\{ a_\nu \}_{\nu \in \N} \sset \D_{\omega, t} (\Delta)$ converges weakly to zero in $\D_{\omega, t} (\Delta)$ then $\{ a_\nu(\lambda) \}_{\nu \in \N}$ converges to zero in $E_\lambda$ for every $\lambda \in \sigma(\Delta)$ (in other words, the inclusion map $\D_{\omega, t} (\Delta) \hookrightarrow \Pi(\Delta)$ is continuous even when $\D_{\omega, t} (\Delta)$ is endowed with the weak topology).

We thus define
\begin{align*}
  \D_{ (\omega ) } (\Delta) \dfn \bigcap_{t > 0} \D_{\omega, t} (\Delta), & \quad \D_{ \{\omega\} } (\Delta) \dfn \bigcup_{t > 0} \D_{\omega, t} (\Delta)
\end{align*}
and endow them with their projective and injective locally convex topologies, respectively: thanks to our previous remarks, they are a FS and a DFS space, respectively, provided $\omega$ is increasing and unbounded. We will see some concrete examples of these spaces later on, but now we show some standard, but convenient, characterizations of them. These characterizations are conditioned on the behavior of a series, namely
\begin{align}
  \sum_{\sigma(\Delta)} e^{-\rho \omega(\lambda)} & < \infty \label{eq:basic_conv_omega1}
\end{align}
where $\rho > 0$.

\begin{Lem} \label{lem:charac_func_ber_rom} \hfill
  \begin{enumerate}
  \item \label{lem:charac_func_ber_rom_it1} If~\eqref{eq:basic_conv_omega1} holds for some $\rho > 0$ then $a \in \Pi(\Delta)$ belongs to $\D_{ (\omega ) } (\Delta)$ if and only if for every $t > 0$ there exists $C > 0$ such that
    \begin{align}
      \| a(\lambda) \|_{L^2(\Omega)} &\leq C e^{-t\omega(\lambda)}, \quad \forall \lambda \in \sigma(\Delta). \label{eq:ineq_gev_t}
    \end{align}
  \item \label{lem:charac_func_ber_rom_it2} If~\eqref{eq:basic_conv_omega1} holds for every $\rho > 0$ then $a \in \Pi(\Delta)$ belongs to $\D_{ \{\omega \} } (\Delta)$ if and only if there exist $t > 0$ and $C > 0$ such that~\eqref{eq:ineq_gev_t} holds.
  \end{enumerate}
  \begin{proof} Necessity is evident in both cases since if $a \in \D_{\omega, t}(\Delta)$ then
    \begin{align*}
      \| a(\lambda) \|_{L^2(\Omega)} &\leq \| a \|_{\D_{\omega, t}(\Delta)} e^{-t\omega(\lambda)}, \quad \forall \lambda \in \sigma(\Delta).
    \end{align*}
    As for the sufficiency:
    \begin{enumerate}
    \item Let $a \in \Pi(\Delta)$ be as in the statement and let $s > 0$ be arbitrary.  Set $t \dfn s + \rho/2$, where $\rho > 0$ is as in~\eqref{eq:basic_conv_omega1}, and let $C > 0$ be such that~\eqref{eq:ineq_gev_t} holds. Then
    \begin{align}
      \sum_{\sigma(\Delta)} e^{2s \omega(\lambda)} \| a(\lambda) \|_{L^2(\Omega)}^2 \leq C^2 \sum_{\sigma(\Delta)} e^{-\rho \omega(\lambda)} < \infty \label{eq:proof_charac_smooth}
    \end{align}
    hence $a \in \D_{\omega, s}(\Delta)$ for every $s > 0$.
    \item Let $a \in \Pi(\Delta)$ be such that~\eqref{eq:ineq_gev_t} holds for some $t > 0$ and $C > 0$. Let $\rho > 0$ be so small that $s \dfn t - \rho / 2$ is positive. Then~\eqref{eq:proof_charac_smooth} holds thanks to~\eqref{eq:basic_conv_omega1}, hence $a \in \D_{\omega, s}(\Delta)$.
    \end{enumerate}
  \end{proof}
\end{Lem}

In order to study regularity properties of operators acting on these spaces (similar to hypoellipticity) it is necessary to provide nice characterizations also for their associated spaces of generalized functions i.e.~their duals. For $t \in \R$ given, $\D_{\omega, t} (\Delta)$ is a Hilbert space so Riesz Theorem tells us that the map
\begin{align*}
  \TR{\flat_t}{v}{\D_{\omega, t}(\Delta)}{\langle \cdot, \overline{v} \rangle_{\D_{\omega, t}(\Delta)}}{\D_{\omega, t}(\Delta)'}
\end{align*}
is a linear isomorphism, and indeed an isometry for the induced Hermitian inner product on $\D_{\omega, t}(\Delta)'$. Moreover, the application
\begin{align*}
  \TR{e^{t \omega(\Delta)}}{a}{\Pi(\Delta)}{\left( e^{t \omega(\lambda)} a(\lambda) \right)_{\lambda \in \sigma(\Delta)}}{\Pi(\Delta)}
\end{align*}
maps $\D_{\omega, t}(\Delta)$ isometrically onto $L^2(\Delta) = \D_{\omega, 0}(\Delta)$, and its inverse is precisely $e^{-t \omega(\Delta)}$.

Denoting the transpose of $e^{t \omega(\Delta)}: \D_{\omega, t}(\Delta) \rarr \D_{\omega, 0}(\Delta)$ by $\transp{e}^{t \omega(\Delta)}: \D_{\omega, 0}(\Delta)' \rarr \D_{\omega, t}(\Delta)'$, one easily proves the commutativity of the diagram
\begin{align}
  \xymatrix{
  \D_{\omega, t}(\Delta) \ar[r]^{e^{t \omega(\Delta)}} \ar[d]_{\flat_t} & \D_{\omega, 0}(\Delta) \ar[d]^{\flat_0} \\
  \D_{\omega, t}(\Delta)'        & \ar[l]^{\transp{e}^{t \omega(\Delta)}} \D_{\omega, 0}(\Delta)' 
                                   } \label{eq:diagram_eqts}
\end{align}
which ensures, in particular, that $\transp{e}^{t \omega(\Delta)}: \D_{\omega, 0}(\Delta)' \rarr \D_{\omega, t}(\Delta)'$ is an isometry.
\begin{Lem} \label{lem:char_dual_ult2} For each $t \in \R$ the bilinear pairing
  \begin{align*}
    (u, v) \in \D_{\omega, t}(\Delta) \times \D_{\omega, -t} (\Delta) &\longmapsto \sum_{\sigma(\Delta)} \langle u(\lambda), \overline{v(\lambda)} \rangle_{L^2(\Omega)} \in \C
  \end{align*}
  gives rise to an isometry $\D_{\omega, -t}(\Delta) \cong \D_{\omega, t}(\Delta)'$, which is precisely $\transp{e}^{t \omega(\Delta)} \circ \flat_0 \circ e^{-t \omega(\Delta)}$.
  \begin{proof} It is clear that the series above converges since for $u \in \D_{\omega, t}(\Delta)$ and $v \in \D_{\omega, -t}(\Delta)$ we have
    \begin{align*}
      \sum_{\sigma(\Delta)} \langle u(\lambda), \overline{v(\lambda)} \rangle_{L^2(\Omega)} = \sum_{\sigma(\Delta)} \left \langle e^{t \omega(\lambda)} u(\lambda), \overline{e^{-t \omega(\lambda)} v(\lambda)} \right \rangle_{L^2(\Omega)} = \left \langle e^{t \omega(\Delta)} u, \overline{e^{-t \omega(\Delta)} v} \right \rangle_{L^2(\Delta)}
    \end{align*}
    from which it is also clear that for every $v \in \D_{\omega, -t}(\Delta)$ the functional
    \begin{align*}
      u \in \D_{\omega, t}(\Delta) &\longmapsto \sum_{\sigma(\Delta)} \langle u(\lambda), \overline{v(\lambda)} \rangle_{L^2(\Omega)} \in \C
    \end{align*}
    is continuous. By~\eqref{eq:diagram_eqts} we have a commutative diagram of isometries
    \begin{align*}
      \xymatrix{
        \D_{\omega, t}(\Delta) \ar[r]^{e^{t \omega(\Delta)}} \ar[d]_{\flat_t} & \D_{\omega, 0}(\Delta) \ar[r]^{e^{t \omega(\Delta)}} \ar[d]^{\flat_0}  & \D_{\omega, -t}(\Delta) \ar[d]^{\flat_{-t}}\\
        \D_{\omega, t}(\Delta)'        & \ar[l]^{\transp{e}^{t \omega(\Delta)}} \D_{\omega, 0}(\Delta)' & \ar[l]^{\transp{e}^{t \omega(\Delta)}} \D_{\omega, -t}(\Delta)'
      }
    \end{align*}
    so the map $\transp{e}^{t \omega(\Delta)} \circ \flat_0 \circ e^{-t \omega(\Delta)}: \D_{\omega, -t}(\Delta) \rarr  \D_{\omega, t}(\Delta)'$ is an isometry: let $\xi \in \D_{\omega, t}(\Delta)'$ and let $v \in \D_{\omega, -t}(\Delta)$ be such that $\transp{e}^{t \omega(\Delta)}  \flat_0  e^{-t \omega(\Delta)} v = \xi$. For every $u \in \D_{\omega, t}(\Delta)$ we have
    \begin{align*}
      \sum_{\sigma(\Delta)} \langle u(\lambda), \overline{v(\lambda)} \rangle_{L^2(\Omega)} = \left \langle e^{t \omega(\Delta)} u, \overline{e^{-t \omega(\Delta)} v} \right \rangle_{L^2(\Delta)} = \left \langle \flat_0 e^{-t \omega(\Delta)} v, e^{t \omega(\Delta)} u \right \rangle = \left \langle \transp{e}^{t \omega(\Delta)}\flat_0 e^{-t \omega(\Delta)} v, u \right \rangle
    \end{align*}
    which equals $\left \langle \xi, u \right \rangle$.
  \end{proof}
\end{Lem}
Now, for $t_+ \geq t$ another tedious (but easy) computation shows that the diagram
\begin{align*}
  \xymatrix{
  \D_{\omega, -t}(\Delta) \ar[r]^{\jmath} \ar[d]_{\cong} & \D_{\omega, -t_+}(\Delta) \ar[d]^{\cong} \\
  \D_{\omega, t}(\Delta)' \ar[r]_{\transp{\imath}}       & \D_{\omega, t_+}(\Delta)' 
                                                      }
\end{align*}
is commutative: here $\imath: \D_{\omega, t_+}(\Delta) \hookrightarrow \D_{\omega, t}(\Delta)$ and $\jmath: \D_{\omega, -t}(\Delta) \hookrightarrow \D_{\omega, -t_+}(\Delta)$ are the respective inclusion maps and the downward arrows are the isometries in Lemma~\ref{lem:char_dual_ult2}.

We define the new spaces
\begin{align*}
  \D'_{ (\omega ) }(\Delta) \dfn \bigcup_{t > 0} \D_{\omega, -t} (\Delta), & \quad \D'_{ \{\omega\} }(\Delta) \dfn \bigcap_{t > 0} \D_{\omega, -t} (\Delta)
\end{align*}
endowed with the locally convex injective and projective topologies, respectively, which turn them into a DFS and a FS space, respectively, if $\omega$ is assumed increasing and unbounded. It follows from Lemma~\ref{lem:char_dual_ult2} and the preceding digression, together with~\cite[Theorems~11 and 12]{kom67}, that in that case the bilinear pairing
\begin{align*}
  (u, v) \in \D_{\star}(\Delta) \times \D'_{\star} (\Delta) &\longmapsto \sum_{\sigma(\Delta)} \langle u(\lambda), \overline{v(\lambda)} \rangle_{L^2(\Omega)} \in \C
\end{align*}
turns $\D_{\star}(\Delta)$ and $\D'_{\star} (\Delta)$ into the strong dual of one another (where $\star = (\omega), \{ \omega \}$).

The proof of the following characterization follows the same lines as in Lemma~\ref{lem:charac_func_ber_rom}. We leave the details for the reader.

\begin{Lem} \label{lem:charac_distr} \hfill
  \begin{enumerate}
  \item \label{lem:charac_distr_it1} If~\eqref{eq:basic_conv_omega1} holds for some $\rho > 0$ then $a \in \Pi(\Delta)$ belongs to $\D'_{ (\omega ) } (\Delta)$ if and only if there exist $t > 0$ and $C > 0$ such that
    \begin{align}
      \| a(\lambda) \|_{L^2(\Omega)} &\leq C e^{t\omega(\lambda)}, \quad \forall \lambda \in \sigma(\Delta). \label{eq:ineq_gevult_t}
    \end{align}
  \item \label{lem:charac_distr_it2} If~\eqref{eq:basic_conv_omega1} holds for every $\rho > 0$ then $a \in \Pi(\Delta)$ belongs to $\D'_{ \{\omega \} } (\Delta)$ if and only if for every $t > 0$ there exists $C > 0$ such that~\eqref{eq:ineq_gevult_t} holds.
  \end{enumerate}
\end{Lem}
\section{Spaces defined by weight functions: examples} \label{sec:examples}

\subsection{Smooth functions}

Recall that $\cinfty(\Omega)$ stands for the space of all complex-valued smooth functions on $\Omega$, which we always consider endowed with its natural Fr{\'e}chet space topology: its dual space, $\D'(\Omega)$, is the space of Schwartz distributions in $\Omega$. Unless explicitly stated otherwise, we consider it endowed with the strong dual topology. The volume form $\dd V$ induced by our Riemannian metric allows us to embed all the classical spaces of functions in $\D'(\Omega)$: we interpret each $f \in L^1(\Omega)$ as a distribution on $\Omega$ by letting it act on a test-function $\phi \in \cinfty(\Omega)$ as
\begin{align}
  \langle f, \phi \rangle &\dfn \int_\Omega f \phi \ \dd V. \label{eq:integ_against}
\end{align}
Because $\Delta$ is real and symmetric, for $f, \phi \in \cinfty(\Omega)$ it holds that
\begin{align*}
  \langle \Delta f, \phi \rangle = \langle \Delta f, \overline{\phi} \rangle_{L^2(\Omega)} = \langle f, \Delta \overline{\phi} \rangle_{L^2(\Omega)} = \langle f, \overline{\Delta \phi} \rangle_{L^2(\Omega)} = \langle f, \Delta \phi \rangle
\end{align*}
which we then take for definition of $\Delta f \in \D'(\Omega)$ when $f \in \D'(\Omega)$. More generally, if $P \in \C[\lambda]$ then
\begin{align*}
  \langle P(\Delta) f, \phi \rangle &= \langle f, P(\Delta) \phi \rangle, \quad \forall \phi \in \cinfty(\Omega),
\end{align*}
which then equals $P(\lambda) \langle f, \phi \rangle$ when $\phi \in E_\lambda$ for some $\lambda \in \sigma(\Delta)$.

Now if $f \in \D'(\Omega)$ then $f|_{E_\lambda} \in E_\lambda^*$, and we denote by $\mathcal{F}_\lambda (f)$ the unique element in $E_\lambda$ that satisfies
\begin{align*}
  \langle \mathcal{F}_\lambda (f), \phi \rangle_{L^2(\Omega)} &= \langle f, \overline{\phi} \rangle, \quad \forall \phi \in E_\lambda.
\end{align*}
Therefore, if we select $\{ \phi^\lambda_j \st 1 \leq j \leq \dim E_\lambda \}$ an orthonormal basis for $E_\lambda$ then
\begin{align*}
  \mathcal{F}_\lambda (f) = \sum_{j = 1}^{d_\lambda} \langle \mathcal{F}_\lambda (f), \phi^\lambda_j \rangle_{L^2(\Omega)} \ \phi^\lambda_j = \sum_{j = 1}^{d_\lambda} \langle f, \overline{\phi^\lambda_j} \rangle \ \phi^\lambda_j
\end{align*}
(where $d_\lambda \dfn \dim E_\lambda$) which agrees with the original definition of $\mathcal{F}_\lambda$ as the orthogonal projection of $L^2(\Omega)$ onto $E_\lambda$. We have thus defined a linear map
\begin{align*}
  \TR{\mathcal{F}}{f}{\D'(\Omega)}{ \left( \mathcal{F}_\lambda(f) \right)_{\lambda \in \sigma(\Delta)}}{\Pi(\Delta)}
\end{align*}
that extends in a natural way the original definition of $\mathcal{F}$. If $P \in \C[\lambda]$ then for every $\phi \in E_\lambda$ we have
\begin{align*}
  \langle \mathcal{F}_\lambda \left( P(\Delta) f \right), \phi \rangle_{L^2(\Omega)} = \langle P(\Delta) f, \overline{\phi} \rangle = P(\lambda) \langle f, \overline{\phi} \rangle = P(\lambda) \langle \mathcal{F}_\lambda (f), \phi \rangle_{L^2(\Omega)} 
\end{align*}
which proves that
\begin{align}
  \mathcal{F}_\lambda \left( P(\Delta) f \right) &= P(\lambda) \mathcal{F}_\lambda (f), \quad \forall \lambda \in \sigma(\Delta). \label{eq:powers_Delta}
\end{align}

For each $k \in \Z_+$ we define
\begin{align*}
  \sob^k(\Omega) &\dfn \left\{ u \in L^2(\Omega) \st (I + \Delta)^k u \in L^2(\Omega) \right\}
\end{align*}
which endowed with the norm
\begin{align*}
  \| u \|_{\sob^k(\Omega)} &\dfn \left\| (I + \Delta)^k u \right\|_{L^2(\Omega)}
\end{align*}
is a Hilbert space. Since $I + \Delta$ is a second-order elliptic LPDO with real-analytic coefficients which is, moreover, injective we have that
\begin{align*}
  \cinfty(\Omega) &= \bigcap_{k \in \Z_+} \sob^k(\Omega)
\end{align*}
as locally convex spaces, where we endow the right-hand side with the projective limit topology: this is a FS space. We define the function $\omega_\infty: \sigma(\Delta) \rarr \R_+$ by
\begin{align*}
  \omega_\infty(\lambda) &\dfn \log (1 + \lambda), \quad \lambda \in \sigma(\Delta),
\end{align*}
which is increasing, unbounded and
\begin{align}
  \sum_{\sigma(\Delta)} e^{-2N \omega_\infty(\lambda)} = \sum_{\sigma(\Delta)} (1 + \lambda)^{-2N} < \infty  \label{eq:series_smooth}
\end{align}
thanks to~\eqref{eq:weyl}: that is, \eqref{eq:basic_conv_omega1} holds for $\omega = \omega_\infty$ when $\rho = 2N$. Take, for each $t \in \R$,
\begin{align*}
  \sob^t(\Delta) \dfn \left\{ a \in \Pi(\Delta) \st \left( (1 + \lambda)^t a(\lambda) \right)_{\lambda \in \sigma(\Delta)} \in L^2(\Delta) \right\} = \D_{\omega_\infty, t} (\Delta)
\end{align*}
which, we recall, is a Hilbert space with norm
\begin{align*}
  \| a \|_{\D_{\omega_\infty, t} (\Delta)} &= \left( \sum_{\sigma(\Delta)} (1 + \lambda)^{2t} \| a(\lambda) \|_{L^2(\Omega)}^2 \right)^{\frac{1}{2}}.
\end{align*}
Applying~\eqref{eq:powers_Delta} to the polynomial $P(\lambda) \dfn (1 + \lambda)^k = e^{k \log(1 + \lambda)}$ one gets a commutative diagram
\begin{align*}
  \xymatrix{
    \sob^k(\Omega) \ar[r]^{\mathcal{F}} \ar[d]_{(I + \Delta)^k} & \Pi(\Delta)  \\
    L^2(\Omega) \ar[r]_{\mathcal{F}}       &  L^2(\Delta) \ar[u]_{e^{-k \omega_\infty(\Delta)}}
  }
\end{align*}
which proves that $\mathcal{F}$ maps $\sob^k(\Omega)$ isometrically onto $e^{-k \omega_\infty(\Delta)}(L^2(\Delta)) = \sob^k(\Delta)$, and one also proves easily that for $f \in \sob^k(\Omega)$ we have~\eqref{eq:abst_fourier_exp} with convergence in $\sob^k(\Omega)$. The space
\begin{align*}
  \cinfty(\Delta) \dfn \bigcap_{k \in \Z_+} \sob^k(\Delta)
\end{align*}
endowed with the corresponding projective limit topology is equal to $\D_{(\omega_\infty)} (\Delta)$ (because $\Z_+$ is cofinal in $\R_+$) and contains the image of $\cinfty(\Omega)$ under $\mathcal{F}$ thanks to our previous remarks. The induced map $\mathcal{F}: \cinfty(\Omega) \rarr \cinfty(\Delta)$ is then injective (because $\mathcal{F}$ is injective in $L^2(\Omega)$) and continuous by the Closed Graph Theorem. It is, moreover, onto: if $a \in \cinfty(\Delta)$ then, for each $k \in \Z_+$:
\begin{align*}
  a \in \sob^k(\Delta) \Longrightarrow \text{$\exists f_k \in \sob^k(\Delta)$ such that $\mathcal{F}(f_k) = a$} \Longrightarrow \mathcal{F}(f_k) = \mathcal{F}(f_j), \quad \forall j, k \in \Z_+.
\end{align*}
By injectivity we have $f_k = f_0$ for all $k \in \Z_+$, hence $f_0 \in \sob^k(\Omega)$ for all $k \in \Z_+$ i.e.~$f_0 \in \cinfty(\Omega)$ and $\mathcal{F}(f_0) = a$. By the Open Mapping Theorem, $\mathcal{F}: \cinfty(\Omega) \rarr \cinfty(\Delta)$ is a topological isomorphism.

Now, the space
\begin{align*}
  \D'(\Delta) \dfn \bigcup_{k \in \Z_+} \sob^{-k}(\Delta),
\end{align*}
which we endow with the corresponding injective limit topology, is equal to $\D'_{(\omega_\infty)} (\Delta)$ as topological vector spaces. We claim that this set is precisely the image of $\D'(\Omega)$ under $\mathcal{F}$, and the induced map $\mathcal{F}: \D'(\Omega) \rarr \D'(\Delta)$ is a topological isomorphism.

Indeed, notice that since $\mathcal{F}: \cinfty(\Omega) \rarr \cinfty(\Delta)$ is a topological isomorphism then so is its transpose $\transp{\mathcal{F}}: \D'(\Delta) \rarr \D'(\Omega)$. For $a \in \D'(\Delta)$ let $f \dfn \transp{\mathcal{F}}(a) \in \D'(\Omega)$, hence $\mathcal{F}(f) \in \Pi(\Delta)$: for $\phi \in E_\lambda$ (for some $\lambda \in \sigma(\Delta)$) we have, by definition of $\mathcal{F}_\lambda(f)$,
\begin{align*}
  \left \langle \mathcal{F}_\lambda(f), \phi \right \rangle_{L^2(\Omega)} = \langle f, \overline{\phi} \rangle = \langle \transp{\mathcal{F}}(a), \overline{\phi} \rangle = \langle a, \mathcal{F}(\overline{\phi}) \rangle = \langle a(\lambda), \phi \rangle_{L^2(\Omega)}
\end{align*}
by means of the duality pairing~\eqref{eq:fund_pairing_pisigma}. We have thus proved that $\mathcal{F}(f) = a$, and since every $f \in \D'(\Omega)$ is of that form we have proved that $\mathcal{F}$ maps $\D'(\Omega)$ into $\D'(\Delta)$. In that sense, we actually showed that $\mathcal{F} \circ \transp{\mathcal{F}}$ is the identity map on $\D'(\Delta)$ i.e.~$\mathcal{F} = (\transp{\mathcal{F}})^{-1}$ is a topological isomorphism.
\subsection{Gevrey functions}

This is very close (and essentially equivalent) to what was done in~\cite{fm95}, but a bit simplified by our use of $L^2$ norms. We start by recalling a little bit about Gevrey theory: given $U \sset \R^N$ an open set and $s \geq 1$, a function $f \in \cinfty(U)$ belongs to the Gevrey space $\gev^s(U)$ if for every compact set $K \sset U$ there exist $h > 0$ and $C > 0$ such that
\begin{align*}
  \sup_K |\del^\alpha f | &\leq C h^{|\alpha|} \alpha!^s, \quad \forall \alpha \in \Z_+^N.
\end{align*}
In particular, $\gev^1(U)$ is simply the space of real-analytic functions in $U$.

The assignment $U \mapsto \gev^s(U)$ is a sheaf that is preserved by real-analytic local diffeomorphisms, hence the notion of Gevrey function can be intrinsically defined in a real-analytic manifold: we denote by $\gev^s(\Omega)$ the space of Gevrey functions on $\Omega$. Since $I + \Delta$ is a second-order elliptic LPDO with real-analytic coefficients in $\Omega$, results in~\cite{kn62},~\cite{bcm79} (see also~\cite[Theorem~2.4, Chapter~8]{lm_nhbv3}) ensure that a function $f \in \cinfty(\Omega)$ belongs to $\gev^s(\Omega)$ if and only if there exist constants $C > 0$ and $h > 0$ such that
\begin{align*}
  \| (I + \Delta)^k f \|_{L^2(\Omega)} &\leq C h^{2k} (2k)!^{s}, \quad \forall k \in \Z_+.
\end{align*}
For each $h > 0$ we define the vector space
\begin{align*}
  \gev^{s,h}(\Omega) &\dfn \left\{ f \in \gev^s(\Omega) \st \| f \|_{\gev^{s,h}(\Omega)} < \infty \right\}
\end{align*}
where
\begin{align*}
  \| f \|_{\gev^{s,h}(\Omega)} &\dfn \sup_k h^{-2k} (2k)!^{-s} \| (I + \Delta)^k f \|_{L^2(\Omega)}
\end{align*}
which turns $\gev^{s,h}(\Omega)$ into a Banach space. The inclusion map $\gev^{s, h}(\Omega) \hookrightarrow \cinfty(\Omega)$ is continuous since for each $k \in \Z_+$ we have
\begin{align*}
  \| f \|_{\sob^k(\Omega)} &\leq h^{2k} (2k)!^{s} \| f \|_{\gev^{s, h}(\Omega)}, \quad \forall f \in \gev^{s, h}(\Omega).
\end{align*}
Moreover, for $h < h_+$ the inclusion map $\gev^{s, h}(\Omega) \hookrightarrow \gev^{s, h_+}(\Omega)$ is clearly continuous, and also compact by standard arguments: the locally convex injetive limit topology on
\begin{align*}
  \gev^{s}(\Omega) &= \bigcup_{h > 0} \gev^{s,h}(\Omega)
\end{align*}
turns it into a DFS space. The function $\omega_s: \sigma(\Delta) \rarr \R_+$ defined by
\begin{align*}
  \omega_s(\lambda) &\dfn (1 + \lambda)^{\frac{1}{2s}}, \quad \lambda \in \sigma(\Delta).
\end{align*}
is increasing and unbounded, and moreover
\begin{align}
  \sum_{\sigma(\Delta)} e^{-\rho \omega_s(\lambda)} = \sum_{\sigma(\Delta)} e^{-\rho (1 + \lambda)^{\frac{1}{2s}}}  < \infty, \quad \forall \rho > 0, \label{eq:series_gevrey}
\end{align}
which follows easily from the next remark (it will be useful later on):
\begin{Rem} \label{rem:compares_smooth_gevrey} Given $s > 0$ and $t, t' \in \R$ with $t < 0$ there exists $C > 0$ such that
  \begin{align*}
    (1 + \lambda)^{t'} &\geq C e^{t(1 + \lambda)^{\frac{1}{s}}}, \quad \forall \lambda \in \sigma(\Delta).
  \end{align*}
\end{Rem}
In particular, \eqref{eq:series_gevrey} follows from~\eqref{eq:series_smooth}. For $t \in \R$ we define
\begin{align*}
  \gev^{s, t}(\Delta) \dfn \left\{ a \in \Pi(\Delta) \st \left( e^{2t (1 + \lambda)^{\frac{1}{2s}}} a(\lambda) \right)_{\lambda \in \sigma(\Delta)} \in L^2(\Delta) \right\} = \D_{\omega_s, t}(\Delta)
\end{align*}
endowed with the Hilbert space norm
\begin{align*}
  \| a \|_{\D_{\omega_s, t}(\Delta)} &= \left( \sum_{\sigma(\Delta)}  e^{2t (1 + \lambda)^{\frac{1}{2s}}} \| a(\lambda) \|_{L^2(\Omega)}^2 \right)^{\frac{1}{2}}.
\end{align*}
As we have seen, the space
\begin{align*}
  \gev^{s}(\Delta) \dfn \bigcup_{t > 0} \gev^{s, t}(\Delta) = \D_{\{\omega_s\}}(\Delta)
\end{align*}
is thus a DFS space when endowed with the locally convex injective limit topology.

We claim that $\mathcal{F}$ maps $\gev^s(\Omega)$ onto $\gev^s(\Delta)$. Indeed, if $f \in \gev^{s}(\Omega)$ we have, for $\lambda \in \sigma(\Delta)$ and $k \in \Z_+$:
\begin{align*}
  \| (1 + \lambda)^k \mathcal{F}_\lambda(f) \|_{L^2(\Omega)} \leq \| (I + \Delta)^k f \|_{L^2(\Omega)} \leq \| f \|_{\gev^{s, h}(\Omega)} h^{2k} (2k)!^s
\end{align*}
hence
\begin{align*}
  \left( \frac{ \left( (1 + \lambda)^{\frac{1}{2s}} (1/h)^{\frac{1}{s}} \right)^{2k}}{(2k)!} \right)^s \| \mathcal{F}_\lambda(f) \|_{L^2(\Omega)} &\leq \| f \|_{\gev^{s, h}(\Omega)}
\end{align*}
or:
\begin{align*}
  \frac{ \left( (1 + \lambda)^{\frac{1}{2s}} (1/h)^{\frac{1}{s}} (1 / \sqrt{2}) \right)^{2k}}{(2k)!} \| \mathcal{F}_\lambda(f) \|_{L^2(\Omega)}^{\frac{1}{s}} &\leq 2^{-k} \| f \|_{\gev^{s, h}(\Omega)}^{\frac{1}{s}}.
\end{align*}
Summing over $k \in \Z_+$ on both sides and using the basic estimate
\begin{align*}
  e^{\frac{r}{2}} &\leq \frac{2}{\sqrt{e}} \sum_{k = 0}^\infty \frac{r^{2k}}{(2k)!}, \quad \forall r \geq 0,
\end{align*}
(which we apply with $r \dfn (1 + \lambda)^{\frac{1}{2s}} (1/h)^{\frac{1}{s}} / \sqrt{2}$) one gets
\begin{align*}
  e^{t(1 + \lambda)^{\frac{1}{2s}}} \| \mathcal{F}_\lambda(f) \|_{L^2(\Omega)} &\leq (4 / \sqrt{e})^s \| f \|_{\gev^{s, h}(\Omega)}, \quad \forall \lambda \in \sigma(\Delta),
\end{align*}
where $t \dfn s (1/h)^{\frac{1}{s}} / (2 \sqrt{2}) > 0$, hence $\mathcal{F}(f) \in \D_{\{\omega_s\}}(\Delta) = \gev^s(\Delta)$ by Lemma~\ref{lem:charac_func_ber_rom}\eqref{lem:charac_func_ber_rom_it2}.

The induced map $\mathcal{F}: \gev^s(\Omega) \rarr \gev^s(\Delta)$ is clearly injective, and also onto: for $t > 0$ let $a \in \gev^{s,t}(\Delta) \sset \cinfty(\Delta)$, hence there exists $f \in \cinfty(\Omega)$ such that $\mathcal{F}(f) = a$. For every $k \in \Z_+$ we have
\begin{align*}
  \| (I + \Delta)^k f \|_{L^2(\Omega)}^2 = \sum_{\sigma(\Delta)} (1 + \lambda)^{2k} \| \mathcal{F}_\lambda(f) \|_{L^2(\Omega)}^2 = \sum_{\sigma(\Delta)} \theta(\lambda)^2 e^{2t (1 + \lambda)^{\frac{1}{2s}}} \| a(\lambda) \|_{L^2(\Omega)}^2
\end{align*}
where we estimate $\theta(\lambda) \dfn e^{-t (1 + \lambda)^{\frac{1}{2s}}} (1 + \lambda)^{k}$ as follows: for $h > 0$ we have
\begin{align*}
  \frac{(1 + \lambda)^k (1/h)^{2k}}{(2k)!^s} = \left( \frac{ \left( (1 + \lambda)^{\frac{1}{2s}} (1/h)^{\frac{1}{s}} \right)^{2k}}{(2k)!} \right)^s \leq e^{s (1/h)^{\frac{1}{s}} (1 + \lambda)^{\frac{1}{2s}}}
\end{align*}
hence
\begin{align*}
  e^{-t (1 + \lambda)^{\frac{1}{2s}}} (1 + \lambda)^{k} &\leq e^{ \left(s (1/h)^{\frac{1}{s}} - t\right) (1 + \lambda)^{\frac{1}{2s}}} h^{2k} (2k)!^s
\end{align*}
which is bounded above by $h^{2k} (2k)!^s$, for every $\lambda \in \sigma(\Delta)$ and $k \in \Z_+$, provided we choose $h$ so big that $s (1/h)^{\frac{1}{s}} - t < 0$. That is:
\begin{align*}
  \| (I + \Delta)^k f \|_{L^2(\Omega)} &\leq \| a \|_{\D_{\omega_s, t}(\Delta)} h^{2k} (2k)!^s, \quad \forall k \in \Z_+,
\end{align*}
i.e.~$f \in \gev^{s, h}(\Omega)$.

Now, continuity and openness of the map $\mathcal{F}: \gev^s(\Omega) \rarr \gev^s(\Delta)$ follow either from very general versions of the Closed Graph Theorem and the Open Mapping Theorem~\cite[pp.~57, 59]{kothe_tvs2} (which hold in the category of DFS spaces), or by a more down-to-earth argument using the estimates we established above. The conclusion is that this map is a topological isomorphism.

The space of Gevrey ultradistributions of order $s \geq 1$ on $\Omega$ is simply the strong dual of $\gev^s(\Omega)$, and we denote it by $\D'_s(\Omega)$. As in the previous section we regard it as a space of generalized functions by means of the pairing ``integration against''~\eqref{eq:integ_against}. The biggest of them is $\D'_1(\Omega)$, usually identified, as a vector space, with the space of hyperfunctions on $\Omega$ (which is compact and oriented). Because every eigenfunction of $\Delta$ is real-analytic (i.e.~belongs to $\gev^1(\Omega)$, and hence is contained in every other $\gev^s(\Omega)$), we can extend $\mathcal{F}$ to $\D'_s(\Omega)$ by the same rule as we did in the previous section for distributions, obtaining thus a linear map $\mathcal{F}: \D'_s(\Omega) \rarr \Pi(\Delta)$ whose image is precisely
\begin{align*}
  \D'_s(\Delta) \dfn \bigcap_{t > 0} \gev^{s, -t}(\Delta) = \D'_{\{ \omega_s\}}(\Delta).
\end{align*}
This, as we have seen, is a FS space that we identify with the strong dual of $\gev^s(\Delta)$; the induced map $\mathcal{F}: \D'_s(\Omega) \rarr \D'_s(\Delta)$ is then a topological isomorphism. The proof of these claims is exactly like the one in the previous section; details are left to the reader.

\section{Regularity} \label{sec:reg}

Inspired by the usual notion of global hypoellipticity for LPDOs (and its analytic and Gevrey versions), as well as the abstract regularity condition~\eqref{eq:abstract_reg}, in this section we study several notions of regularity for (systems of) $\Delta$FDOs. Throughout this section, $\omega: \sigma(\Delta) \rarr \R_+$ is increasing and unbounded.
\begin{Def} \label{def:almost_hyp} For $\star = (\omega), \{ \omega \}$, we say that a matrix of $\Delta$FDOs $P = (P_{ij})_{n \times m}$ is \emph{almost $\D_{\star}$ globally hypoelliptic} -- or $(\mathrm{AGHE})_{\star}$ for short -- if for every $u \in \D'_{\star}(\Delta)^m$ such that $Pu \in \D_{\star}(\Delta)^n$ there exists $v \in \ker P$ such that $u - v \in \D_{\star}(\Delta)^m$.
\end{Def}

Despite the fancy name -- which we introduce only to simplify a few statements -- this idea is probably not new. The word ``almost'' comes from the fact that we are just checking regularity of the operator where it matters (for our purposes): it tells us nothing about the regularity of objects that lie in the kernel of the operator, as the usual notions of (global) hypoellipticity do. Indeed, it is clear that these more familiar notions of hypoellipticity can be recovered from our ``almost'' versions by imposing the appropriate regularizing property on the kernel of $P$.

These seem to be the correct notions of regularity for systems, in our present context, especially the ones we are interested in -- and we shall devote this entire section to convince the reader of this claim. One should think of $\Delta$FDOs as operators with ``constant coefficients'' in $\Omega$, although this notion has no rigorous meaning. As such, we expect their regularity properties (i.e.~various notions of global hypoellipticity) to be governed by certain inequalities involving their symbols. These ideas are going to guide our approach in what follows.

The results presented in this section are in the same spirit as those contained in the book by Wallach~\cite{wallach_haohs} and his earlier article with Greenfield~\cite{gw73} (although both of them deal only with smooth regularity). Our approach is more closely related to the latter -- working on a general compact manifold instead of a Lie group --, but we stress the validity of the results for systems (just like the former). An advantage of doing so is avoiding the hassle of Representation Theory\footnote{See Remark~\ref{rem:rep_theory}.} when it is not needed, which is definitely the case of our results in their final form; an obvious drawback is that we deal only with globally trivial systems, and the passage to more general situations (operators acting between invariant vector bundles on compact homogeneous spaces) is not entirely immediate.

We start by showing how an estimate on the symbol of a $\Delta$FDO $P$ implies a strong regularity property.
\begin{Prop} \label{prop:basic_ineq_ult} Suppose that for some $s \in \R$ and $C > 0$ we have
  \begin{align}
    \| \widehat{P}(\lambda) \phi \|_{L^2(\Omega)} &\geq C e^{s \omega(\lambda)}, \quad \forall \phi \in \ker \widehat{P}(\lambda)^\bot, \ \| \phi \|_{L^2(\Omega)} = 1, \ \forall \lambda \in \sigma(\Delta). \label{eq:hyp_system_ult}
  \end{align}
  Then for every $u \in \Pi(\Delta)$ there exists $v \in \ker P$, independent of $s$ and $C$, such that
  \begin{align*}
    \forall t \in \R, \ Pu \in \D_{\omega, t}(\Delta) &\Longrightarrow u - v \in \D_{\omega, s + t}(\Delta)
  \end{align*}
  and in this context the following estimate holds:
  \begin{align}
    \| u - v \|_{\D_{\omega, s + t}(\Delta)} &\leq C^{-1} \| P u \|_{\D_{\omega, t}(\Delta)}. \label{eq:hyp_system_ult2}
  \end{align}
  \begin{proof} Suppose that~\eqref{eq:hyp_system_ult} holds and let $u \in \Pi(\Delta)$.  For each $\lambda \in \sigma(\Delta)$ we write $u(\lambda) = v(\lambda) + w(\lambda)$ where $v(\lambda) \in \ker \widehat{P}(\lambda)$ and $w(\lambda) \in \ker \widehat{P}(\lambda)^\bot$. For every $\lambda \in \sigma(\Delta)$ we have
    \begin{align*}
      e^{s \omega(\lambda)} \| u(\lambda) - v(\lambda) \|_{L^2(\Omega)} = e^{s \omega(\lambda)} \| w(\lambda) \|_{L^2(\Omega)} \leq C^{-1} \| \widehat{P}(\lambda) w(\lambda) \|_{L^2(\Omega)} = C^{-1} \| \widehat{P}(\lambda) u(\lambda) \|_{L^2(\Omega)}.
    \end{align*}
    Squaring both sides, multiplying them by $e^{2t \omega(\lambda)}$ and summing over $\sigma(\Delta)$ yields~\eqref{eq:hyp_system_ult2}.
  \end{proof}
\end{Prop}

Next, we list some instances when our basic estimate~\eqref{eq:hyp_system_ult} does not hold.
\begin{Prop} \label{prop:neg_basic_ineq_ult} Let $P$ be a non-zero $\Delta$FDO.
  \begin{enumerate}
  \item \label{prop:neg_basic_ineq_ult_it1} Assume that~\eqref{eq:basic_conv_omega1} holds for some $\rho > 0$.  If for every $s \in \R$ the estimate~\eqref{eq:hyp_system_ult} does not hold for any $C > 0$ then there exist:
    \begin{enumerate}
    \item \label{prop:neg_basic_ineq_ult_it1a} $u \in \D'_{(\omega)}(\Delta)$ such that $Pu \in \D_{(\omega)}(\Delta)$ but $u + v \notin \D_{(\omega)}(\Delta)$ for every $v \in \ker P$;
    \item \label{prop:neg_basic_ineq_ult_it1b} $u \in \Pi(\Delta)$ such that $Pu \in \D'_{(\omega)}(\Delta)$ but $u + v \notin \D'_{(\omega)}(\Delta)$ for every $v \in \ker P$.
    \end{enumerate}
  \item \label{prop:neg_basic_ineq_ult_it2} Assume that~\eqref{eq:basic_conv_omega1} holds for every $\rho > 0$. If for some $s < 0$ the estimate~\eqref{eq:hyp_system_ult} does not hold for any $C > 0$ then there exist:
    \begin{enumerate}
    \item \label{prop:neg_basic_ineq_ult_it2a} $u \in \D'_{\{\omega\}}(\Delta)$ such that $Pu \in \D_{\{\omega\}}(\Delta)$ but $u + v \notin \D_{\{\omega\}}(\Delta)$ for every $v \in \ker P$;
    \item \label{prop:neg_basic_ineq_ult_it2b} $u \in \Pi(\Delta)$ such that $Pu \in \D'_{\{\omega\}}(\Delta)$ but $u + v \notin \D'_{\{\omega\}}(\Delta)$ for every $v \in \ker P$.
    \end{enumerate}
  \end{enumerate}
  \begin{proof} \hfill
    \begin{enumerate}
    \item 
      \begin{enumerate}
      \item We shall prove more: that for each $s \in \R$ there exists $u \in \D_{\omega, s}(\Delta)$ such that $Pu \in \D_{(\omega)} (\Delta)$ but $u + v \notin \D_{\omega, s + \rho/2} (\Delta)$ no matter what $v \in \ker P$ is: the conclusion follows by taking $s = 0$. For each $\nu \in \N$ we apply the hypothesis to $s \dfn -\nu$: there are $\lambda_\nu \in \sigma(\Delta)$ and $\phi_\nu \in \ker \widehat{P}(\lambda)^\bot$ with $\| \phi_\nu \|_{L^2(\Omega)} = 1$ such that
        \begin{align*}
          \| \widehat{P}(\lambda_\nu) \phi_\nu \|_{L^2(\Omega)} &< 2^{-\nu} e^{-\nu \omega(\lambda_\nu)}.
        \end{align*}
        We can assume w.l.o.g.~that $\{\lambda_\nu\}_{\nu \in \N}$ is increasing. Let $s \in \R$ and define $u \in \Pi(\Delta)$ by
        \begin{align*}
          u(\lambda) &\dfn
          \begin{cases}
            e^{- (s + \rho/2) \omega(\lambda_\nu)} \phi_\nu, &\text{if $\lambda = \lambda_\nu$ for some $\nu$} \\
            0, &\text{otherwise}.
          \end{cases}
        \end{align*}
        For $t \in \R$ we have
        \begin{align*}
          \sum_{\sigma(\Delta)} e^{2(s + t) \omega(\lambda)} \| u (\lambda) \|_{L^2(\Omega)}^2 = \sum_{\nu = 1}^\infty e^{2(s + t) \omega(\lambda_\nu)} \frac{ \| \phi_\nu \|_{L^2(\Omega)}^2}{ e^{2(s + \rho/2) \omega(\lambda_\nu)}} = \sum_{\nu = 1}^\infty e^{2(t - \rho/2) \omega(\lambda_\nu)} 
        \end{align*}
        therefore $u \in \D_{\omega, s}(\Delta) \setminus \D_{\omega, s + \rho/2}(\Delta)$ again. However
        \begin{align*}
          \sum_{\sigma(\Delta)} e^{2t \omega(\lambda)} \| \widehat{P}(\lambda)u (\lambda) \|_{L^2(\Omega)}^2 &\leq \sum_{\nu = 1}^\infty  2^{-2 \nu} e^{2(t - s - \rho/2 - \nu) \omega(\lambda_\nu)}
        \end{align*}
        is finite for every $t \in \R$, that is, $Pu \in \D_{(\omega)}(\Delta)$. Moreover, if $v \in \ker P$ then for $\lambda \in \sigma(\Delta)$ we have $v(\lambda) \in \ker \widehat{P}(\lambda)$ so
        \begin{align*}
          \| u(\lambda) + v(\lambda) \|_{L^2(\Omega)}^2 = \| u(\lambda) \|_{L^2(\Omega)}^2 + \| v(\lambda) \|_{L^2(\Omega)}^2  \geq \| u(\lambda) \|_{L^2(\Omega)}^2
        \end{align*}
        and simple computation shows that $u + v \notin \D_{\omega, t}(\Delta)$ whenever $u \notin \D_{\omega, t}(\Delta)$, for every $t \in \R$.
      \item For each $\nu \in \N$ take $\lambda_\nu \in \sigma(\Delta)$ and $\phi_\nu \in \ker \widehat{P}(\lambda)^\bot$ as in the previous item, but now define $u \in \Pi(\Delta)$ by
    \begin{align*}
      u(\lambda) &\dfn
      \begin{cases}
        e^{\nu \omega(\lambda_\nu)} \phi_\nu, &\text{if $\lambda = \lambda_\nu$ for some $\nu$} \\
        0, &\text{otherwise}.
      \end{cases}
    \end{align*}
    On the one hand we have, for all $\nu \in \N$,
    \begin{align*}
      \| \widehat{P}(\lambda_\nu) u(\lambda_\nu) \|_{L^2(\Omega)} = e^{\nu \omega(\lambda_\nu)} \| \widehat{P}(\lambda_\nu) \phi_\nu \|_{L^2(\Omega)} \leq 2^{-\nu}
    \end{align*}
    which by Lemma~\ref{lem:charac_distr}\eqref{lem:charac_distr_it1} implies that $Pu \in \D'_{(\omega)}(\Delta)$. On the other hand, by that very lemma we conclude that $u \notin \D'_{(\omega)}(\Delta)$ for
    \begin{align*}
      \| u(\lambda_\nu) \|_{L^2(\Omega)} &= e^{\nu \omega(\lambda_\nu)}, \quad \forall \nu \in \N,
    \end{align*}
    and $\omega$ is positive, increasing and unbounded.
      \end{enumerate}
    \item
      \begin{enumerate}
      \item The hypothesis implies that for every $\nu \in \N$ there exist $\lambda_\nu \in \sigma(\Delta)$ and $\phi_\nu \in \ker \widehat{P}(\lambda)^\bot$ with $\| \phi_\nu \|_{L^2(\Omega)} = 1$ such that
        \begin{align*}
          \| \widehat{P}(\lambda_\nu) \phi_\nu \|_{L^2(\Omega)} &< e^{s \omega(\lambda_\nu)}
        \end{align*}
        and as usual we can assume that $\{\lambda_\nu\}_{\nu \in \N}$ is increasing. Define $u \in \Pi(\Delta)$ by
        \begin{align*}
          u(\lambda) &\dfn
          \begin{cases}
            \phi_\nu, &\text{if $\lambda = \lambda_\nu$ for some $\nu$} \\
            0, &\text{otherwise}.
          \end{cases}
        \end{align*}
        Clearly $u \in \D'_{\{\omega\}}(\Delta)$ by Lemma~\ref{lem:charac_distr}\eqref{lem:charac_distr_it2}, but obviously $u \notin L^2(\Delta)$ since $\| u(\lambda) \|_{L^2(\Omega)} \leq 1$ with equality for infinitely many $\lambda \in \sigma(\Delta)$. On the other hand, $Pu \in \D_{\{\omega\}}(\Delta)$ by Lemma~\ref{lem:charac_func_ber_rom}\eqref{lem:charac_func_ber_rom_it2}, whereas $u + v \notin L^2(\Delta)$ for every $v \in \ker P$ thanks to previous arguments.
      \item For each $\nu \in \N$ take $\lambda_\nu \in \sigma(\Delta)$ and $\phi_\nu \in \ker \widehat{P}(\lambda)^\bot$ as in the previous item, but now define $u \in \Pi(\Delta)$ by
    \begin{align*}
      u(\lambda) &\dfn
      \begin{cases}
        e^{-s\omega(\lambda_\nu)} \phi_\nu, &\text{if $\lambda = \lambda_\nu$ for some $\nu$} \\
        0, &\text{otherwise}.
      \end{cases}
    \end{align*}
    On the one hand we have, for all $\nu \in \N$,
    \begin{align*}
      \| \widehat{P}(\lambda_\nu) u(\lambda_\nu) \|_{L^2(\Omega)} = e^{-s \omega(\lambda_\nu)} \| \widehat{P}(\lambda_\nu) \phi_\nu \|_{L^2(\Omega)} \leq 1
    \end{align*}
    which by Lemma~\ref{lem:charac_distr}\eqref{lem:charac_distr_it2} implies that $Pu \in \D'_{\{\omega\}}(\Delta)$. On the other hand, by that very lemma we conclude that $u \notin \D'_{\{\omega\}}(\Delta)$ for
    \begin{align*}
      \| u(\lambda_\nu) \|_{L^2(\Omega)} &= e^{-s \omega(\lambda_\nu)}, \quad \forall \nu \in \N,
    \end{align*}
    and $\omega$ is positive, increasing and unbounded (take $t \dfn -s/2$ in that lemma).
      \end{enumerate}
    \end{enumerate}
  \end{proof}
\end{Prop}

From Propositions~\ref{prop:basic_ineq_ult} and~\ref{prop:neg_basic_ineq_ult} (only parts~\eqref{prop:neg_basic_ineq_ult_it1a},\eqref{prop:neg_basic_ineq_ult_it2a} -- parts~\eqref{prop:neg_basic_ineq_ult_it1b},\eqref{prop:neg_basic_ineq_ult_it2b} will be needed later on) it follows immediately that:
\begin{Thm} \label{thm:hyp_charac_estim} Let $P$ be a non-zero $\Delta$FDO.
  \begin{enumerate}
  \item \label{thm:hyp_charac_estim_it1} Assume that~\eqref{eq:basic_conv_omega1} holds for some $\rho > 0$. Then $P$ is $(\mathrm{AGHE})_{(\omega)}$ if and only if there exist $s \in \R$ and $C > 0$ such that~\eqref{eq:hyp_system_ult} holds. In that case, for every $u \in \Pi(\Delta)$ there exists $v \in \ker P$ such that
    \begin{align*}
      \forall t \in \R, \ Pu \in \D_{\omega, t}(\Delta) &\Longrightarrow u - v \in \D_{\omega, s + t}(\Delta)
    \end{align*}
    and, in particular, if $Pu \in \D_{(\omega)}(\Delta)$ then $u - v \in \D_{(\omega)}(\Delta)$.
  \item \label{thm:hyp_charac_estim_it2} Assume that~\eqref{eq:basic_conv_omega1} holds for every $\rho > 0$. Then $P$ is $(\mathrm{AGHE})_{\{\omega\}}$ if and only if for every $s < 0$ there exists $C > 0$ such that~\eqref{eq:hyp_system_ult} holds. In that case, for every $u \in \Pi(\Delta)$ there exists $v \in \ker P$ such that
    \begin{align*}
      \forall t \in \R, \ \forall s < 0, \ Pu \in \D_{\omega, t}(\Delta) &\Longrightarrow u - v \in \D_{\omega, s + t}(\Delta)
    \end{align*}
    and, in particular, if $Pu \in \D_{\{\omega\}}(\Delta)$ then $u - v \in \D_{\{\omega\}}(\Delta)$.
  \end{enumerate}
\end{Thm}
\begin{Cor} If~\eqref{eq:basic_conv_omega1} holds for every $\rho > 0$ then $(\mathrm{AGHE})_{\{\omega\}}$ implies $(\mathrm{AGHE})_{(\omega)}$.
\end{Cor}
Now recalling the spaces introduced in Section~\ref{sec:examples}:
\begin{Cor} Let $P$ be a non-zero $\Delta$FDO. If $P$ is almost $\cinfty$-globally hypoelliptic then $P$ is almost $\gev^s$-globally hypoelliptic for every $s \geq 1$. The former is, moreover, implied by the property that $P$ has a closed range in $L^2(\Delta)$ (regarded as an unbounded, densely defined operator).
  \begin{proof} By Theorem~\ref{thm:hyp_charac_estim}\eqref{thm:hyp_charac_estim_it2}, in order to prove that $P$ is almost $\gev^s$-globally hypoelliptic we must check the following: given $t < 0$ there exists $C > 0$ such that
    \begin{align*}
      \| \widehat{P}(\lambda) \phi \|_{L^2(\Omega)} &\geq C e^{t(1 + \lambda)^{\frac{1}{2s}}}, \quad \forall \phi \in \ker \widehat{P}(\lambda)^\bot, \ \| \phi \|_{L^2(\Omega)} = 1, \ \forall \lambda \in \sigma(\Delta).
    \end{align*}
    The hypothesis of almost $\cinfty$-global hypoellipticity implies, thanks to Theorem~\ref{thm:hyp_charac_estim}\eqref{thm:hyp_charac_estim_it1}, the existence of $t' \in \R$ and $C' > 0$ such that
    \begin{align*}
      \| \widehat{P}(\lambda) \phi \|_{L^2(\Omega)} &\geq C' (1 + \lambda)^{t'}, \quad \forall \phi \in \ker \widehat{P}(\lambda)^\bot, \ \| \phi \|_{L^2(\Omega)} = 1, \ \forall \lambda \in \sigma(\Delta).
    \end{align*}
    The conclusion then follows from Remark~\ref{rem:compares_smooth_gevrey}. Moreover, the existence of such $t'$ follows, for instance, from~\eqref{eq:closedL2} (take $t' \dfn 0$ in this case) -- which characterizes closedness of the range of $P$ in $L^2(\Delta)$ according to Proposition~\ref{prop:closedL2}, thus proving the second part of the statement.
  \end{proof}
\end{Cor}

Now we turn our attention to those $\Delta$FDOs that preserve our spaces.
\begin{Thm} \label{thm:closed_iff_hypo} Let $P$ be a non-zero $\Delta$FDO.
  \begin{enumerate}
  \item \label{thm:closed_iff_hypo_it1} Suppose that~\eqref{eq:basic_conv_omega1} holds for some $\rho > 0$ and that $P \D_{(\omega)}(\Delta) \sset \D_{(\omega)}(\Delta)$. Then the induced map $P: \D_{(\omega)}(\Delta) \rarr \D_{(\omega)}(\Delta)$ has a closed range if and only if $P$ is $(\mathrm{AGHE})_{(\omega)}$.
  \item \label{thm:closed_iff_hypo_it2} Suppose that~\eqref{eq:basic_conv_omega1} holds for every $\rho > 0$ and that $P \D_{\{\omega\}}(\Delta) \sset \D_{\{\omega\}}(\Delta)$. Then the induced map $P: \D_{\{\omega\}}(\Delta) \rarr \D_{\{\omega\}}(\Delta)$ has a closed range if and only if $P$ is $(\mathrm{AGHE})_{\{\omega\}}$.
  \end{enumerate}
  \begin{proof}  First of all, notice that in both cases the map $P: \D_{\star}(\Delta) \rarr \D_{\star}(\Delta)$ is continuous: its graph is closed (by continuity of the map $P: \Pi(\Delta) \rarr \Pi(\Delta)$) and the Closed Graph Theorem~\cite[p.~57]{kothe_tvs2} applies.

    \begin{enumerate}
    \item We claim that
      \begin{align*}
        F &\dfn \{ u \in \D_{(\omega)}(\Delta) \st u(\lambda) \in \ker \widehat{P}(\lambda)^\bot, \ \forall \lambda \in \sigma(\Delta) \}
      \end{align*}
      is a closed subspace of $\D_{(\omega)}(\Delta)$: if $\{ u_\nu \}_{\nu \in \N} \sset F$ and $u \in \D_{(\omega)}(\Delta)$ are such that $u_\nu \to u$ in $\D_{(\omega)}(\Delta)$ then for each $\lambda \in \sigma(\Delta)$ we have $u_\nu(\lambda) \to u(\lambda)$, and $\ker \widehat{P}(\lambda)^\bot$ is closed in $E_\lambda$.
      
      Suppose that $G \dfn \ran \{ P: \D_{(\omega)}(\Delta) \rarr \D_{(\omega)}(\Delta) \}$ is a closed subspace of $\D_{(\omega)}(\Delta)$. Then $P: F \rarr G$ is a bijective continuous linear map between Fr{\'e}chet spaces, hence admits a continuous inverse $P^{-1}: G \rarr F$ by the Open Mapping Theorem. By definition of the projective limit topology, $\{ \| \cdot \|_{\D_{\omega, t}(\Delta)} \st t > 0 \}$ is a system of seminorms that generates the topology of $\D_{(\omega)}(\Delta)$. Because $F, G \sset \D_{(\omega)}(\Delta)$ are closed subspaces, the continuity of $P^{-1}: G \rarr F$ reads as follows: for every $t > 0$ there exist $s > 0$ and $C > 0$ such that
      \begin{align*}
        \| P^{-1} v \|_{\D_{\omega, t}(\Delta)} &\leq C \| v \|_{\D_{\omega, s}(\Delta)}, \quad \forall v \in G
      \end{align*}
      (recall that finitely many ``Sobolev norms'' in the right-hand side can be dominated by a single one) which can then be restated as
      \begin{align*}
        \| u \|_{\D_{\omega, t}(\Delta)} &\leq C \| Pu \|_{\D_{\omega, s}(\Delta)}, \quad \forall u \in F.
      \end{align*}
      In particular, for $\lambda \in \sigma(\Delta)$ let $\phi \in \ker \widehat{P}(\lambda)^\bot$. Then $\phi \in F$ (as an element of $\Pi(\Delta)$) and
      \begin{align*}
        e^{t \omega(\lambda)} \| \phi \|_{L^2(\Omega)} = \| \phi \|_{\D_{\omega, t}(\Delta)} \leq C \| P\phi \|_{\D_{\omega, s}(\Delta)} = C e^{s \omega(\lambda)} \| \widehat{P}(\lambda) \phi \|_{L^2(\Omega)}
      \end{align*}
      which implies that
      \begin{align*}
        \| \widehat{P}(\lambda) \phi \|_{L^2(\Omega)} &\geq C^{-1} e^{(t - s) \omega(\lambda)} \| \phi \|_{L^2(\Omega)}, \quad \forall \phi \in \ker \widehat{P}(\lambda)^\bot, \ \forall \lambda \in \sigma(\Delta).
      \end{align*}
      It follows from Theorem~\ref{thm:hyp_charac_estim}\eqref{thm:hyp_charac_estim_it1} that $P$ is $(\mathrm{AGHE})_{(\omega)}$.
      
      For the converse, assume that $P$ is $(\mathrm{AGHE})_{(\omega)}$ and let $\{ u_\nu \}_{\nu \in \N} \sset \D_{(\omega)}(\Delta)$ and $w \in \D_{(\omega)}(\Delta)$ be such that $Pu_\nu \to w$ in $\D_{(\omega)}(\Delta)$. Since this convergence also holds in $\Pi(\Delta)$ we have
      \begin{align*}
        \widehat{P}(\lambda) u_\nu(\lambda) &\to w(\lambda), \quad \forall \lambda \in \sigma(\Delta)
      \end{align*}
      and since $\ran \widehat{P}(\lambda)$ is always closed in $E_\lambda$ we can find $u \in \Pi(\Delta)$ such that $Pu = w$, simply by solving $\widehat{P}(\lambda) u(\lambda) = w(\lambda)$ for each $\lambda \in \sigma(\Delta)$ individually. By Theorem~\ref{thm:hyp_charac_estim}\eqref{thm:hyp_charac_estim_it1} there exists $v \in \ker P$ such that $u - v \in \D_{(\omega)}(\Delta)$ since $Pu \in \D_{(\omega)}(\Delta)$. But $P(u - v) = Pu = w$, which proves that $w$ actually belongs to $\ran \{ P: \D_{(\omega)}(\Delta) \rarr \D_{(\omega)}(\Delta) \}$.
    \item By~\cite[Theorem~2.5]{araujo18} the range of $P: \D_{\{\omega\}}(\Delta) \rarr \D_{\{\omega\}}(\Delta)$ is closed if and only if for every $t > 0$ there exists $t' > 0$ such that for every $u \in \D_{\{\omega\}}(\Delta)$ we have
      \begin{align}
        Pu \in \D_{\omega, t}(\Delta) &\Longrightarrow \text{$\exists v \in \ker \{ P: \D_{\{\omega\}}(\Delta) \rarr \D_{\{\omega\}}(\Delta)\}$ such that $u - v \in \D_{\omega, t'}(\Delta)$}. \label{eq:closed_r}
      \end{align}
      This is what we prove by assuming $P$ to be $(\mathrm{AGHE})_{\{\omega\}}$. For $u \in \D_{\{\omega\}}(\Delta)$ let $v \in \ker P$ be as in Theorem~\ref{thm:hyp_charac_estim}\eqref{thm:hyp_charac_estim_it2}: we have
      \begin{align*}
        \forall t > 0, \ Pu \in \D_{\omega, t}(\Delta) &\Longrightarrow u - v \in \D_{\omega, t/2}(\Delta).
      \end{align*}
      Since $v \in \D_{\{\omega\}}(\Delta)$ we have proved~\eqref{eq:closed_r} with $t' \dfn t/2$.
      
      For the converse, assume that $P$ is not $(\mathrm{AGHE})_{\{\omega\}}$: by Theorem~\ref{thm:hyp_charac_estim}\eqref{thm:hyp_charac_estim_it2} there exists $s < 0$ such that estimate~\eqref{eq:hyp_system_ult} holds for no $C > 0$, and by Proposition~\ref{prop:neg_basic_ineq_ult}\eqref{prop:neg_basic_ineq_ult_it2a} one can find $\tilde{u} \in \D'_{\{\omega\}}(\Delta)$ such that $P\tilde{u} \in \D_{\{\omega\}}(\Delta)$ while $\tilde{u} + \tilde{v} \notin L^2(\Delta)$ for every $\tilde{v} \in \ker P$. Assume by contradiction that the range of $P: \D_{\{\omega\}}(\Delta) \rarr \D_{\{\omega\}}(\Delta)$ is closed. Let $t > 0$ be such that $P\tilde{u} \in \D_{\omega, t}(\Delta)$ and let $t' > 0$ be as in~\eqref{eq:closed_r}. Let $u \dfn e^{-\frac{t'}{2} \omega(\Delta)} \tilde{u} \in \D_{\{\omega\}}(\Delta)$. Then $Pu = e^{-\frac{t'}{2} \omega(\Delta)} P\tilde{u} \in \D_{\omega, t + t'/2}(\Delta) \sset \D_{\omega, t}(\Delta)$ hence by assumption there exists $v \in \ker \{ P: \D_{\{\omega\}}(\Delta) \rarr \D_{\{\omega\}}(\Delta)\}$ such that $u - v \in \D_{\omega, t'}(\Delta)$. But then
      \begin{align*}
        \tilde{u} - e^{\frac{t'}{2} \omega(\Delta)} v = e^{\frac{t'}{2} \omega(\Delta)} (u - v) \in \D_{\omega, t'/2}(\Delta) \sset L^2(\Delta)
      \end{align*}
      which contradicts our hypothesis since
      \begin{align*}
        P \left( e^{\frac{t'}{2} \omega(\Delta)} v \right) = e^{\frac{t'}{2} \omega(\Delta)} Pv = 0.
      \end{align*}
      We conclude that the range of $P: \D_{\{\omega\}}(\Delta) \rarr \D_{\{\omega\}}(\Delta)$ cannot be closed.
    \end{enumerate}
  \end{proof}
\end{Thm}
We stress that all the results above, as well as their proofs, are valid for systems. One can also define the following ``distributional'' versions of Definition~\ref{def:almost_hyp}.
\begin{Def} \label{def:almost_hyp_ult} For $\star = (\omega), \{ \omega \}$, we say that a matrix of $\Delta$FDOs $P = (P_{ij})_{n \times m}$ is \emph{almost $\D'_{\star}$ globally hypoelliptic} -- or $(\mathrm{AGHE})'_{\star}$ for short -- if for every $u \in \Pi(\Delta)^m$ such that $Pu \in \D'_{\star}(\Delta)^n$ there exists $v \in \ker P$ such that $u - v \in \D'_{\star}(\Delta)^m$.
\end{Def}
By adapting the ideas in the proofs above (also using parts~\eqref{prop:neg_basic_ineq_ult_it1b},\eqref{prop:neg_basic_ineq_ult_it2b} in Proposition~\ref{prop:neg_basic_ineq_ult}) one has:
\begin{Thm} \label{thm:closed_iff_hypo_ult} Let $P$ be a non-zero $\Delta$FDO.
  \begin{enumerate}
  \item \label{thm:closed_iff_hypo_ult_it1} Suppose that~\eqref{eq:basic_conv_omega1} holds for some $\rho > 0$. The following properties are equivalent.
    \begin{enumerate}
    \item \label{thm:closed_iff_hypo_ult_it1a} $P$ is $(\mathrm{AGHE})'_{(\omega)}$.
    \item \label{thm:closed_iff_hypo_ult_it1b} $P$ is $(\mathrm{AGHE})_{(\omega)}$.
    \item \label{thm:closed_iff_hypo_ult_it1c} In case $P \D'_{(\omega)}(\Delta) \sset \D'_{(\omega)}(\Delta)$ the induced map $P: \D'_{(\omega)}(\Delta) \rarr \D'_{(\omega)}(\Delta)$ has a closed range.
    \end{enumerate}
    \item \label{thm:closed_iff_hypo_ult_it2} Suppose that~\eqref{eq:basic_conv_omega1} holds for every $\rho > 0$. The following properties are equivalent.
    \begin{enumerate}
    \item \label{thm:closed_iff_hypo_ult_it2a} $P$ is $(\mathrm{AGHE})'_{\{\omega\}}$.
    \item \label{thm:closed_iff_hypo_ult_it2b} $P$ is $(\mathrm{AGHE})_{\{\omega\}}$.
    \item \label{thm:closed_iff_hypo_ult_it2c} In case $P \D'_{\{\omega\}}(\Delta) \sset \D'_{\{\omega\}}(\Delta)$ the induced map $P: \D'_{\{\omega\}}(\Delta) \rarr \D'_{\{\omega\}}(\Delta)$ has a closed range.
    \end{enumerate}
  \end{enumerate}
  \begin{proof} The only more delicate part of the proof is the equivalence between~\eqref{thm:closed_iff_hypo_ult_it1a} and~\eqref{thm:closed_iff_hypo_ult_it1c}, which is similar to the proof of Theorem~\ref{thm:closed_iff_hypo}\eqref{thm:closed_iff_hypo_it2} but nevertheless we sketch it.

    By~\cite[Theorem~2.5]{araujo18} the range of $P: \D'_{(\omega)}(\Delta) \rarr \D'_{(\omega)}(\Delta)$ is closed if and only if for every $t > 0$ there exists $t' > 0$ such that for every $u \in \D'_{(\omega)}(\Delta)$ we have
    \begin{align}
      Pu \in \D_{\omega, -t}(\Delta) &\Longrightarrow \text{$\exists v \in \ker \{ P: \D'_{(\omega)}(\Delta) \rarr \D'_{(\omega)}(\Delta) \}$ such that $u - v \in \D_{\omega, -t'}(\Delta)$}. \label{eq:closed_rprime}
    \end{align}
    Assume this, but also, by contradiction, that $P$ is not $(\mathrm{AGHE})'_{(\omega)}$. It follows easily from Theorem~\ref{thm:hyp_charac_estim}\eqref{thm:hyp_charac_estim_it1} -- which also settles the equivalence with item~\eqref{thm:closed_iff_hypo_ult_it1b} above -- that estimate~\eqref{eq:hyp_system_ult} does not hold for any $s \in \R$ and $C > 0$. Thanks to the proof of Proposition~\ref{prop:neg_basic_ineq_ult}\eqref{prop:neg_basic_ineq_ult_it1a} (in which we take $s \dfn -t' - \rho/2$) there exists $u \in \D_{\omega, -t' - \rho/2}(\Delta) \sset \D'_{(\omega)}(\Delta)$ such that $Pu \in \D_{(\omega)} (\Delta) \sset \D_{\omega, -t}(\Delta)$ but $u + v \notin \D_{\omega, s + \rho/2} (\Delta) = \D_{\omega, -t'} (\Delta)$ for every $v \in \ker P$. This contradicts~\eqref{eq:closed_rprime}, hence $P$ is $(\mathrm{AGHE})'_{(\omega)}$.

    For the converse, assume $P$ to be $(\mathrm{AGHE})'_{(\omega)}$: by Proposition~\ref{prop:neg_basic_ineq_ult}\eqref{prop:neg_basic_ineq_ult_it1b} estimate~\eqref{eq:hyp_system_ult} holds for some $s \in \R$ and $C > 0$. Let $t > 0$ be given and $u \in \D'_{(\omega)}(\Delta)$ be such that $Pu \in \D_{\omega, -t}(\Delta)$. By Theorem~\ref{thm:hyp_charac_estim}\eqref{thm:hyp_charac_estim_it1} there exists $v \in \ker P$ such that $u - v \in \D_{\omega, s - t}(\Delta)$. By taking $t' \dfn \max \{ t - s, 0 \}$ in~\eqref{eq:closed_rprime} we conclude that $P: \D'_{(\omega)}(\Delta) \rarr \D'_{(\omega)}(\Delta)$ has a closed range.
  \end{proof}
\end{Thm}
\section{Regularity and cohomology of complexes} \label{sec:reg_com}

Next, we apply our results in the previous section to complexes. For an increasing and unbounded function $\omega: \sigma(\Delta) \rarr \R_+$, let $\mathscr{V}$ denote one of the spaces: $\D_{(\omega)}(\Delta)$, $\D_{\{\omega\}}(\Delta)$, $\D'_{(\omega)}(\Delta)$ or $\D'_{\{\omega\}}(\Delta)$. Let $P = (P_{ij})_{n \times m}$ and $Q = (Q_{jk})_{m \times r}$ be two matrices of $\Delta$FDOs forming a differential complex~\eqref{eq:complexPQ} such that $Q$ maps $\mathscr{V}^r$ into $\mathscr{V}^m$, which is then mapped by $P$ into $\mathscr{V}^n$, forming a new complex~\eqref{eq:complexPQ_V}: as we have seen, both of these induced maps are continuous thanks to the Closed Graph Theorem.

In order to simplify the statements below we shall say that $Q$ is \emph{almost $\mathscr{V}$ globally hypoelliptic} if~\eqref{eq:abstract_reg} holds. Thanks to the results in the previous section (Theorem~\ref{thm:hyp_charac_estim}, Definition~\ref{def:almost_hyp_ult}), there is no loss of generality in doing so (provided, of course, we assume some convenient extra hypotheses on~\eqref{eq:basic_conv_omega1}, which we shall always do).

\begin{Lem} \label{lem:findim_cohom_closedr} If $\mathcal{H}_{P,Q}(\mathscr{V})$ is finite dimensional then $Q: \mathscr{V}^r \rarr \mathscr{V}^m$ has a closed range.
  \begin{proof} Let $\mathscr{V}$ be either $\D_{(\omega)}(\Delta)$ or $\D'_{\{\omega\}}(\Delta)$ (resp.~$\D_{\{\omega\}}(\Delta)$ or $\D'_{(\omega)}(\Delta)$): $\mathscr{V}$ is a FS space (resp.~DFS space). For simplicity we write
    \begin{align*}
      X \dfn \ker \{P: \mathscr{V}^m \rarr \mathscr{V}^n \}, \quad  Y \dfn \ker \{Q: \mathscr{V}^r \rarr \mathscr{V}^m \}, \quad  Z \dfn \ran \{Q: \mathscr{V}^r \rarr \mathscr{V}^m \}.
    \end{align*}
    The former two, as closed subspaces of the respective ambient spaces, are FS spaces (resp.~DFS spaces) for the respective subspace topologies. We must establish closedness of $Z$ in $\mathscr{V}^m$.

    To say that $n \dfn \dim(X/Z)$ is finite means that there exist $u_1, \ldots, u_n \in X$ whose classes modulo $Z$ form a basis for $X/Z$: if $M$ stands for their linear span in $X$ then $X = M \oplus Z$. Indeed, notice first that $M \cap Z = \{0 \}$, for any non-trivial linear combination of $u_1, \ldots, u_n$ that belongs to $Z$ would be mapped by the projection map to a null combination of their classes, so the scalars in the linear combination must be zero. Moreover, the same argument ensures that $u_1, \ldots, u_n$ are linearly independent, thus forming a basis for $M$ which, therefore, has dimension $n$. Also, the class of an element $u \in X$ modulo $Z$ must either be zero (hence $u \in Z$) or be a linear combination of the classes of $u_1, \ldots, u_n$, meaning that $u$ differs from $Z$ by an element in $M$, thus proving our claim.

    Now define the map
    \begin{align*}
      \TR{T}{([w], v)}{(\mathscr{V}^r / Y) \times M}{Qw + v}{X}
    \end{align*}
    which is continuous because $Q$ is, and bijective by the previous argument. By the Open Mapping Theorem for Fr{\'e}chet spaces (resp.~De Wilde's~\cite[p.~59]{kothe_tvs2}), $T$ is an isomorphism and as such maps closed subspaces to closed subspaces, hence $Z = T((\mathscr{V}^r / Y) \times \{0\})$ is closed in $X$.
  \end{proof}
\end{Lem}

\begin{Thm} \label{thm:main} Let $\mathscr{V}$ be either $\D_{(\omega)}(\Delta)$ or $\D'_{(\omega)}(\Delta)$ (resp.~$\D_{\{\omega\}}(\Delta)$ or $\D'_{\{\omega\}}(\Delta)$) and assume that~\eqref{eq:basic_conv_omega1} holds for some (resp.~all) $\rho > 0$. Then $\mathcal{H}_{P,Q}(\mathscr{V})$ is finite dimensional if and only if $\mathcal{H}_{P,Q}(\Sigma(\Delta))$ is finite dimensional and $Q$ is almost $\mathscr{V}$ globally hypoelliptic, and in that case
  \begin{align*}
    \dim \mathcal{H}_{P,Q} (\mathscr{V}) &= \sum_{\sigma(\Delta)} \dim \left( \frac{\ker \widehat{P}(\lambda)}{\ran \widehat{Q}(\lambda)} \right).
  \end{align*}
  \begin{proof} If $\mathcal{H}_{P,Q}(\mathscr{V})$ is finite dimensional then by Proposition~\ref{prop:findimV} so is $\mathcal{H}_{P,Q}(\Sigma(\Delta))$, and by Lemma~\ref{lem:findim_cohom_closedr} we have that $Q: \mathscr{V}^r \rarr \mathscr{V}^m$ has a closed range -- which, by Theorems~\ref{thm:closed_iff_hypo} or~\ref{thm:closed_iff_hypo_ult}, ensures that $Q$ is almost $\mathscr{V}$ globally hypoelliptic. The identity between dimensions is precisely~\eqref{eq:basic_eq_dim}.

    On the other hand, if $\dim \mathcal{H}_{P,Q}(\Sigma(\Delta)) < \infty$ and $Q$ is almost $\mathscr{V}$ globally hypoelliptic then by Lemma~\ref{lem:reg_impl_inject_cohom} (as well as the digression that precedes it) all the natural maps~\eqref{eq:natural_maps_cohomology} are injective, with the endpoints of the same dimension according to Corollary~\ref{cor:cond_fin_dim}, thus yielding finite dimensionality of $\mathcal{H}_{P,Q}(\mathscr{V})$.
  \end{proof}
\end{Thm}
\section{Applications} \label{sec:app}

Let $G$ be a compact, connected Lie group, $\gr{g}$ its Lie algebra and $\gr{v} \sset \C \gr{g}$ a Lie subalgebra, which we regard as a left-invariant involutive structure on $G$ as in the Introduction: we refer the reader to it for notation. There, we were denoting by $\LL_1, \ldots, \LL_n, \MM_1, \ldots, \MM_m$ a basis of $\C \gr{g}$ such that $\LL_1, \ldots, \LL_n$ is a basis of $\gr{v}$; and by $\tau_1, \ldots, \tau_n, \zeta_1, \ldots, \zeta_m \in \C \gr{g}^*$ we meant the corresponding dual basis. In this section, we insist in the identification~\eqref{eq:isom_dprime1} i.e.~sections of $\Lambda^{p,q}$ are interpreted as sections of the subbundle of $\wedge^{p + q} \C T^* G$ spanned by the partial frame of left-invariant forms~\eqref{eq:bases_pq} -- hence true $(p + q)$-forms on $G$. This means that every global section $u$ of $\Lambda^{p,q}$ has a ``canonical representative'' of the form~\eqref{eq:urepglobal}, which we regard as $u$ itself.

Once more we endow $G$ with an $\ad$-invariant metric (which, as recalled in the Introduction, always exists since $G$ is compact) and denote by $\Delta$ its Laplace-Beltrami operator. Thus $\Delta$ commutes with every left-invariant vector field (actually with every left-invariant differential operator), which leads us to conclude that when one expresses the differential operator $\dd': \Lambda^{p,q} \rarr \Lambda^{p,q + 1}$ associated with $\gr{v}$ (or, more precisely, with the left-invariant involutive structure $\VV \sset \C T G$) as a matrix of operators w.r.t.~the bases~\eqref{eq:bases_pq} -- as we did in the Introduction -- its entries commute with $\Delta$ i.e.~we may regard $\dd'$ as a matrix of $\Delta$FDOs. We are then allowed to apply all the results deduced in the previous sections to it.

In order to better carry out this program we shall, for each $\lambda \in \sigma(\Delta)$, denote by $\mathscr{E}_\lambda^{p,q}$ the space of all $u \in \cinfty(G; \Lambda^{p,q})$ such that $u_{IJ} \in E_\lambda$ for every ordered multi-indices $I,J$. These spaces are clearly finite dimensional, and we have induced maps $\dd': \mathscr{E}_\lambda^{p,q} \rarr \mathscr{E}_\lambda^{p,q + 1}$; the induced cohomology space is precisely $H^{p,q}_\VV(G; E_\lambda)$, which is always finite dimensional.

\subsection{Immediate consequences} \label{sec:immediate}

The main conclusions from the previous sections, applied to this situation, can be then read as follows (for $p \in \{0, \ldots, m\}$ and $q \in \{0, \ldots, n - 1\}$):
\begin{enumerate}
\item The map $\dd': \cinfty(G; \Lambda^{p,q}) \rarr \cinfty(G; \Lambda^{p,q + 1})$ has a closed range if and only if the same holds for the map $\dd': \D'(G; \Lambda^{p,q}) \rarr \D'(G; \Lambda^{p,q + 1})$. 
\item For each $s \geq 1$, the map $\dd': \gev^s(G; \Lambda^{p,q}) \rarr \gev^s(G; \Lambda^{p,q + 1})$ has a closed range if and only if the same holds for the map $\dd': \D_s'(G; \Lambda^{p,q}) \rarr \D_s'(G; \Lambda^{p,q + 1})$.
\item If the (unbounded, densely defined) map $\dd': L^2(G; \Lambda^{p,q}) \rarr L^2(G; \Lambda^{p,q + 1})$ has a closed range then so does the map $\dd': \cinfty(G; \Lambda^{p,q}) \rarr \cinfty(G; \Lambda^{p,q + 1})$.
\item If the map $\dd': \cinfty(G; \Lambda^{p,q}) \rarr \cinfty(G; \Lambda^{p,q + 1})$ has a closed range then the same is true for $\dd': \gev^s(G; \Lambda^{p,q}) \rarr \gev^s(G; \Lambda^{p,q + 1})$ for every $s \geq 1$. If, moreover, either the smooth cohomology space $H^{p, q + 1}_\VV(G; \cinfty(G))$ or the Gevrey cohomology space $H^{p, q + 1}_\VV(G; \gev^s(G))$ is finite dimensional then both of them are isomorphic to
  \begin{align*}
    \bigoplus_{\sigma(\Delta)} \frac{\ker \left \{ \dd': \mathscr{E}_\lambda^{p,q + 1} \longrightarrow \mathscr{E}_\lambda^{p,q + 2} \right \} }{ \ran \left \{ \dd': \mathscr{E}_\lambda^{p,q} \longrightarrow \mathscr{E}_\lambda^{p,q + 1} \right \} } &= \bigoplus_{\sigma(\Delta)} H^{p, q + 1}_\VV (G; E_\lambda).
  \end{align*}
  This can be seen as a global version of~\cite[Theorem~5.1]{cc11}, or a version of~\cite[Corollary~2.2]{greenfield72} for systems.
\end{enumerate}

\subsection{Left-invariant cohomology} \label{sec:licoho}

Let $G$ be a Lie group. A de Rham cohomology class on $G$ is said to be \emph{left-invariant} if it has some left-invariant representative: recall that a differential form $f$ on $G$ is left-invariant if $(L_x)^* f = f$ for all $x \in G$. When $G$ is compact and connected, every de Rham cohomology class is left-invariant~\cite{ce48}. This can be seen as a result about de Rham's structure $\VV \dfn \C T G$, which is obviously left-invariant with underlying Lie algebra $\gr{v} \dfn \C \gr{g}$, and raises the following question:
\begin{quote}
  Let $G$ be compact and connected and $\VV$ be a left-invariant involutive structure on $G$. Given a bidegree $(p,q)$, when is every cohomology class in $H^{p,q}_\VV(G; \cinfty(G))$ left-invariant?
\end{quote}
Leaving aside proper definitions for a moment, in this section we prove:
\begin{Thm} \label{thm:main2} Let $G$ be a compact and connected Lie group and $\VV$ be a left-invariant involutive structure on $G$. For any given bidegree $(p,q)$, if $H^{p,q}_\VV(G; \cinfty(G))$ is finite dimensional then every cohomology class in $H^{p,q}_\VV(G; \cinfty(G))$ left-invariant if and only if $H^{p,q}_\VV(G; E_\lambda) = 0$ for every $\lambda \in \sigma(\Delta) \setminus 0$.
\end{Thm}
From it we derive the following interesting result.
\begin{Cor} \label{cor:ell_ss} If $\VV$ is elliptic and $\gr{v} \sset \C \gr{g}$, its underlying Lie algebra, is semisimple then every cohomology class in $H^{0,q}_\VV(G; \cinfty(G))$ is left-invariant, for every $q$.
  \begin{proof} To say that $\VV$ is elliptic means that $\VV_x + \bar{\VV}_x = \C T_x G$ for every $x \in G$, or, equivalently, that $\gr{v} + \bar{\gr{v}} = \C \gr{g}$. One can prove that in that case the $\dd'$ differential complex is an elliptic complex, and it is well-known that $H^{p,q}_\VV(G; \cinfty(G))$ is then finite dimensional for every bidegree $(p,q)$. This fact does not depend on the group structure of $G$ and relies solely on the compactness of $G$ as a smooth manifold, and to prove it one can, for instance, either adapt the proof of a classical result in Complex Analysis~\cite[Theorem~VIII,~A19]{gr_afscv} -- where the local solvability of $\dd'$ (by ellipticity cf.~\cite[Chapter~VI]{treves_has}) plays a role similar to the existence of Stein neighborhoods of each point in the complex case --, or write down a Hodge-like theory for the $\dd'$ complex.

    As for the vanishing of $H^{0,q}_\VV(G; E_\lambda)$ for $\lambda \neq 0$ it will be proved below (Lemma~\ref{lem:ell_ss}). The conclusion then follows from Theorem~\ref{thm:main2}.
  \end{proof}
\end{Cor}

\begin{Rem} Our intention with Corollary~\ref{cor:ell_ss} is simply to provide a small (potential) application of our Theorem~\ref{thm:main2}. We stress that we do not discuss in the present work the existence of elliptic and semisimple subalgebras of $\C \gr{g}$ on a compact Lie group. 
\end{Rem}

We start with the actual definitions. Let $u \in \cinfty(G; \Lambda^{p,q})$ be as in~\eqref{eq:urepglobal}. We will say that $u$ is \emph{left-invariant} if it is left-invariant regarded as a $(p + q)$-form on $G$. Notice that for each $x \in G$ we have
\begin{align*}
  (L_x)^* u = \psum_{|I| = p} \psum_{|J| = q} (u_{IJ} \circ L_x) \ (L_x)^* (\zeta_I \wedge \tau_J) = \psum_{|I| = p} \psum_{|J| = q} (u_{IJ} \circ L_x) \ \zeta_I \wedge \tau_J
\end{align*}
from which we conclude, comparing this expression with~\eqref{eq:urepglobal}, that $u$ is left-invariant if and only if
\begin{align*}
  u_{IJ} \circ L_x &= u_{IJ}, \quad \forall x \in G, \ \forall I,J,
\end{align*}
i.e.~each coefficient $u_{IJ}$ is a constant function. In particular, it is clear that this notion of left-invariance does not depend on the choice of basis~\eqref{eq:canonical_frame}, since any two such bases of $\C \gr{g}$ differ by a linear automorphism of $\C \gr{g}$ that preserves $\gr{v}$. Moreover, when $G$ is compact and connected we have that $E_0$ is precisely the space of constant functions on $G$, hence the space of all left-invariant elements of $\cinfty(G; \Lambda^{p,q})$ equals $\mathscr{E}^{p,q}_0$ and the cohomology space of such objects is $H^{p,q}_\VV(G; E_0)$. The inclusion map $\mathscr{E}^{p,q}_0 \hookrightarrow \cinfty(G; \Lambda^{p,q})$ induces a natural map
\begin{align}
  H^{p,q}_\VV(G; E_0) &\longrightarrow H^{p,q}_\VV(G; \cinfty(G)). \label{eq:map_li} 
\end{align}
Left-invariance of every cohomology class in $H^{p,q}_\VV(G; \cinfty(G))$ is equivalent to \emph{surjectivity} of~\eqref{eq:map_li}.
\begin{proof}[Proof of Theorem~\ref{thm:main2}] Thanks to our previous digression we must discuss the surjectivity of~\eqref{eq:map_li}, which we do in the general setup of systems of $\Delta$FDOs just like in Section~\ref{sec:fdos}. Let again $P = (P_{ij})_{n \times m}$ and $Q = (Q_{jk})_{m \times r}$ be matrices of $\Delta$FDOs forming a differential complex~\eqref{eq:complexPQ} such that $\mathscr{V} \dfn \cinfty(\Delta) \cong \cinfty(\Omega)$ is ``invariant'' by this system, so~\eqref{eq:complexPQ_V} holds. Then the inclusion maps $E_0^m \hookrightarrow \Sigma(\Delta)^m \hookrightarrow \cinfty(\Delta)^m$ induce natural maps
  \begin{align*}
    \mathcal{H}_{P,Q}(E_0) \longrightarrow \mathcal{H}_{P,Q}(\Sigma(\Delta)) \longrightarrow \mathcal{H}_{P,Q}(\cinfty(\Delta))
  \end{align*}
  and we claim that if $\dim \mathcal{H}_{P,Q}(\cinfty(\Delta)) < \infty$ then their composition is onto if and only if $\mathcal{H}_{P,Q}(E_\lambda) = 0$ for every $\lambda \neq 0$. Indeed, finite dimensionality of $\mathcal{H}_{P,Q}(\cinfty(\Delta))$ implies by Theorem~\ref{thm:main} and Corollary~\ref{cor:isoms_abstract} that the second map $\mathcal{H}_{P,Q}(\Sigma(\Delta)) \rightarrow \mathcal{H}_{P,Q}(\cinfty(\Delta))$ is an isomorphism, hence surjectivity of the composition above is equivalent to surjectivity of the first map $\mathcal{H}_{P,Q}(E_0) \rightarrow \mathcal{H}_{P,Q}(\Sigma(\Delta))$. But by Lemma~\ref{lem:algebraic_isoms_delta} we have a natural isomorphism
  \begin{align*}
    \mathcal{H}_{P,Q}(\Sigma(\Delta)) &\cong \bigoplus_{\sigma(\Delta)} \mathcal{H}_{P,Q}(E_\lambda)
  \end{align*}
  (now with both sides finite dimensional) and inspection on its proof shows that the first map $\mathcal{H}_{P,Q}(E_0) \rightarrow \mathcal{H}_{P,Q}(\Sigma(\Delta))$ can be then regarded as the canonical injection of a factor into a finite direct sum, and this is obviously onto if and only if all the remaining factors (i.e.~$\mathcal{H}_{P,Q}(E_\lambda)$ with $\lambda \neq 0$) are zero.
\end{proof}

Now, in order to finish the proof of Corollary~\ref{cor:ell_ss} we prove the following:
\begin{Lem} \label{lem:ell_ss} If $\VV$ is elliptic and $\gr{v}$ is semisimple then $H^{0,q}_\VV (G; E_\lambda) = 0$ for every $\lambda \in \sigma(\Delta) \setminus 0$.
\end{Lem}
Its proof uses some elementary tools from Lie algebra cohomology theory, to which we dedicate our final section.

\subsection{Lie algebra cohomology}

Let $\gr{g}$ be a real or complex Lie algebra. A \emph{$\gr{g}$-module} (or a \emph{representation of $\gr{g}$}) is a vector space $\mathscr{V}$ together with a homomorphism of Lie algebras $\gr{g} \rarr \End(\mathscr{V})$, the latter endowed with the usual commutator bracket of linear maps. Two important examples are the following: $\gr{g}$ itself is a $\gr{g}$-module via adjoint action; and given $\mathscr{V}$ and $\mathscr{W}$ two $\gr{g}$-modules, the space $C^r(\mathscr{V}; \mathscr{W})$ of all $r$-multilinear maps from $\mathscr{V}$ to $\mathscr{W}$ ($r \in \Z_+$) carries a natural $\gr{g}$-module structure: $C^0(\mathscr{V}; \mathscr{W}) = \mathscr{W}$, while for $r \geq 1$ one defines
\begin{align*}
  (\vv{X} \omega) (x_1, \ldots, x_r) &\dfn \vv{X} \left( \omega (x_1, \ldots, x_r) \right) + \sum_{j = 1}^r (-1)^j \omega(\vv{X} x_j, x_1, \ldots, \hat{x}_j, \ldots, x_r), \quad x_1, \ldots, x_r \in \mathscr{V}, 
\end{align*}
for $\omega \in C^r(\mathscr{V}; \mathscr{W})$ and $\vv{X} \in \gr{g}$.

For $\mathscr{V}$ a $\gr{g}$-module and $r \in \Z_+$ we define a homomorphism of $\gr{g}$-modules $\dd: C^r(\gr{g}; \mathscr{V}) \rarr C^{r + 1}(\gr{g}; \mathscr{V})$ in the following way: given $\omega \in C^r(\gr{g}; \mathscr{V})$ let
\begin{multline*}
  (\dd \omega)(\vv{X}_0, \ldots, \vv{X}_r) \dfn \\ \dfn \sum_{j = 0}^r (-1)^j \vv{X}_j \left( \omega(\vv{X}_0, \ldots, \hat{\vv{X}}_j, \ldots, \vv{X}_r) \right) + \sum_{j < k} (-1)^{j + k} \omega([\vv{X}_j, \vv{X}_k], \vv{X}_0, \ldots, \hat{\vv{X}}_j, \ldots, \hat{\vv{X}}_k, \ldots, \vv{X}_r)
\end{multline*}
where $\vv{X}_0, \ldots, \vv{X}_r \in \gr{g}$. Then $\dd^2 = 0$ i.e.~we have a differential complex of $\gr{g}$-modules and homomorphisms, called the \emph{Chevalley-Eilenberg complex}~\cite{ce48}, whose cohomology $\gr{g}$-modules we denote by $H^*(\gr{g}; \mathscr{V})$.

One can also define cohomology relative to a Lie subalgebra $\gr{h} \sset \gr{g}$. We define for $p,q \in \Z_+$
\begin{align*}
  \gr{N}^{p,q}_{\gr{h}}(\gr{g}; \mathscr{V}) &\dfn
  \begin{cases}
    C^q(\gr{g}; \mathscr{V}), &\text{if $p = 0$}, \\
    \{ \omega \in C^{p + q}(\gr{g}; \mathscr{V}) \st \text{$\omega(\vv{X}_1, \ldots, \vv{X}_{p + q}) = 0$ if $\vv{X}_1, \ldots, \vv{X}_{q + 1} \in \gr{h}$} \}, &\text{if $p \geq 1$},
  \end{cases}
\end{align*}
which, in either case, is a $\gr{h}$-submodule of $C^{p + q}(\gr{g}; \mathscr{V})$. The following relations hold:
\begin{enumerate}
\item $\gr{N}^{p + 1,q - 1}_{\gr{h}}(\gr{g}; \mathscr{V}) \sset \gr{N}^{p,q}_{\gr{h}}(\gr{g}; \mathscr{V})$ and
\item $\dd \gr{N}^{p,q}_{\gr{h}}(\gr{g}; \mathscr{V}) \sset \gr{N}^{p,q + 1}_{\gr{h}}(\gr{g}; \mathscr{V})$.
\end{enumerate}
Hence
\begin{align*}
  \gr{U}^{p,q}_{\gr{h}} (\gr{g}; \mathscr{V}) &\dfn \gr{N}^{p,q}_{\gr{h}}(\gr{g}; \mathscr{V}) / \gr{N}^{p + 1,q - 1}_{\gr{h}}(\gr{g}; \mathscr{V})
\end{align*}
is a $\gr{h}$-module and $\dd: \gr{N}^{p,q}_{\gr{h}}(\gr{g}; \mathscr{V}) \rarr \gr{N}^{p,q + 1}_{\gr{h}}(\gr{g}; \mathscr{V})$ (now regarded as a homomorphism of $\gr{h}$-modules) induces a homomorphism of $\gr{h}$-modules on the quotients $\dd': \gr{U}^{p,q}_{\gr{h}} (\gr{g}; \mathscr{V}) \rarr \gr{U}^{p,q + 1}_{\gr{h}} (\gr{g}; \mathscr{V})$ such that $(\dd')^2 = 0$ i.e.~for each fixed $p \in \Z_+$ we have a differential complex of $\gr{h}$-modules and homomorphisms, whose cohomology $\gr{h}$-modules we denote by $H^{p,*}_{\gr{h}} (\gr{g}; \mathscr{V})$. Easy but tedious computations~\cite[Theorem~2]{hs53} show that the map $\Phi: \gr{N}^{p,q}_{\gr{h}}(\gr{g}; \mathscr{V}) \rarr C^q(\gr{h}; C^p( \gr{g} / \gr{h}; \mathscr{V} ))$ defined by the formula
\begin{align*}
  (\Phi \omega)(\vv{X}_1, \ldots, \vv{X}_q)(\vv{Y}_1 + \gr{h}, \ldots, \vv{Y}_p + \gr{h}) &\dfn \omega(\vv{X}_1, \ldots, \vv{X}_q, \vv{Y}_1, \ldots, \vv{Y}_p)
\end{align*}
(where $\omega \in \gr{N}^{p,q}_{\gr{h}}(\gr{g}; \mathscr{V})$ and $\vv{X}_1, \ldots, \vv{X}_q \in \gr{h}$ and $\vv{Y}_1, \ldots, \vv{Y}_p \in \gr{g}$) is an isomorphism of $\gr{h}$-modules that descends to an isomorphism of $\gr{h}$-modules in cohomology i.e.
\begin{align}
  H^{p,q}_{\gr{h}} (\gr{g}; \mathscr{V}) &\cong H^q(\gr{h}; C^p( \gr{g} / \gr{h}; \mathscr{V} )), \quad \forall p,q \in \Z_+. \label{eq:hs_isom}
\end{align}

Back to our Lie group $G$, and again $\gr{g}$ its Lie algebra, an obvious $\C \gr{g}$-module is $\cinfty(G)$: vector fields act on smooth functions as differential operators. Since we know how to make sense of this action on other spaces of (generalized) functions -- always keeping in mind that a left-invariant vector field is automatically real-analytic -- we can work just as well with the space of hyperfunctions $\D'_1(G)$, as well as many of its $\C \gr{g}$-submodules. Moreover, since (as mentioned in the Introduction) left-invariant vector fields commute with $\Delta$ -- the Laplace-Beltrami operator associated with some $\ad$-invariant metric on $G$ -- we have that each one of its eigenspaces $E_\lambda$ is a $\C \gr{g}$-module. One can also make vector fields act on differential forms and currents as Lie derivatives, so the spaces of such objects are also $\C \gr{g}$-modules. 

Let $r \in \Z_+$. To each $f \in \cinfty(G; \wedge^r \C T^* G)$ we assign an element $\Psi f$ in $C^r(\C \gr{g}; \cinfty(G))$ defined by
\begin{align*}
  (\Psi f)(\vv{Y}_1, \ldots, \vv{Y}_r) &\dfn f(\vv{Y}_1, \ldots, \vv{Y}_r), \quad \vv{Y}_1, \ldots, \vv{Y}_r \in \C \gr{g}. 
\end{align*}
This yields a linear map $\Psi: \cinfty(G; \wedge^r \C T^* G) \rarr C^r(\C \gr{g}; \cinfty(G))$, which we claim to be a linear isomorphism: introducing bases
\begin{align}
  \vv{X}_1, \ldots, \vv{X}_N \in \C \gr{g}, & \quad \chi_1, \ldots, \chi_N \in \C \gr{g}^* \label{eq:general_bases}
\end{align}
dual to each other, every $f \in \cinfty(G; \wedge^r \C T^* G)$ can be written in a unique fashion as
\begin{align}
  f &= \psum_{|I| = r} f_I \ \chi_I \label{eq:rep_form_coeff}
\end{align}
where, for each ordered multi-index $I = (i_1, \ldots, i_r)$: $f_I \dfn f(\vv{X}_{i_1}, \ldots, \vv{X}_{i_r})$ and $\chi_I \dfn \chi_{i_1} \wedge \cdots \wedge \chi_{i_r}$. If $\Psi f = 0$ then clearly $f_I = 0$ for every $I$, hence $f = 0$ i.e.~$\Psi$ is injective. On the other hand, given $F \in C^r(\C \gr{g}; \cinfty(G))$ define $f_I \dfn F(\vv{X}_{i_1}, \ldots, \vv{X}_{i_r}) \in \cinfty(G)$ for each ordered multi-index $I$: we then assemble back $f \in \cinfty(G; \wedge^r \C T^* G)$ taking~\eqref{eq:rep_form_coeff} as its definition, and since for every ordered multi-index $I$ we have $(\Psi f)(\vv{X}_{i_1}, \ldots, \vv{X}_{i_r}) = f_I = F(\vv{X}_{i_1}, \ldots, \vv{X}_{i_r})$ and, moreover, $\vv{X}_1, \ldots, \vv{X}_N$ is a basis of $\C \gr{g}$ and both $F$ and $\Psi f$ are multilinear, this implies that $\Psi f = F$, hence $\Psi$ is onto. One also checks by hand that $\Psi$ is natural w.r.t.~differentials (exterior derivative on one side and the Chevalley-Eilenberg differential on the other) and induces isomorphism in cohomology:
\begin{align}
  H^r_{\mathrm{dR}} (G; \cinfty(G)) &\cong H^r(\C \gr{g}; \cinfty(G)). \label{eq:laiso_dR}
\end{align}

Similarly relative to a subalgebra $\gr{v} \sset \C \gr{g}$, for $r = p + q$  one checks that $\Psi$ maps $\cinfty(G; \T'^{p,q})$ onto $\gr{N}^{p,q}_{\gr{v}}(\C \gr{g}; \cinfty(G))$ (now it is best to express everything in terms of~\eqref{eq:canonical_frame} and its dual basis, which is a particular choice of~\eqref{eq:general_bases}) and descends to quotients as an isomorphism
\begin{align*}
  \Psi_*: \cinfty(G; \Lambda^{p,q}) \longrightarrow \gr{U}^{p,q}_{\gr{v}}(\C \gr{g}; \cinfty(G))
\end{align*}
which again commutes with differentials (now the $\dd'$ associated to $\VV$ and the relative Chevalley-Eilenberg $\dd'$ associated to the subalgebra $\gr{v}$) and induces isomorphism in cohomology
\begin{align}
  H^{p,q}_{\VV} (G; \cinfty(G)) &\cong H^{p,q}_{\gr{v}} (\C \gr{g}; \cinfty(G)) \label{eq:laiso_VV}
\end{align}
thus generalizing our previous ``absolute'' isomorphism~\eqref{eq:laiso_dR}, which is a special case of~\eqref{eq:laiso_VV} with $\gr{v} \dfn \C \gr{g}$. Of course $\cinfty(G)$ plays no distinguished role here, and the same argument applies to any $\C \gr{g}$-submodule $\mathscr{V} \sset \D'_1(G)$ yielding
\begin{align}
  H^{p,q}_{\VV} (G; \mathscr{V}) &\cong H^{p,q}_{\gr{v}} (\C \gr{g}; \mathscr{V}). \label{eq:laiso_VV2}
\end{align}
The relevant $\C \gr{g}$-modules $\mathscr{V}$ in our previous and forthcoming arguments are $\cinfty(G)$, $\D'(G)$, $\gev^s(G)$, $\D'_s(G)$ (for $s \geq 1$) and $E_\lambda$ (for $\lambda \in \sigma(\Delta)$).
\begin{proof}[Proof of Lemma~\ref{lem:ell_ss}] Let $\lambda \in \sigma(\Delta)$. Composing~\eqref{eq:laiso_VV2} with~\eqref{eq:hs_isom} yields
  \begin{align*}
    H^{0,q}_{\VV} (G; E_\lambda) \cong H^{0,q}_{\gr{v}} (\C \gr{g}; E_\lambda) \cong H^q(\gr{v}; C^0(\C \gr{g} / \gr{v}; E_\lambda)) = H^q(\gr{v}; E_\lambda).
  \end{align*}
  Since $\gr{v}$ is by hypothesis semisimple and $E_\lambda$ is a finite dimensional $\gr{v}$-module, it follows from Whitehead's Vanishing Theorem (see e.g.~\cite[Theorem~5.7.33]{hn_sglg}) that $H^q(\gr{v}; E_\lambda) = 0$ (for all $q$) provided
  \begin{align*}
    E_\lambda^{\gr{v}} &\dfn \{ \phi \in E_\lambda \st \LL \phi = 0, \ \forall \LL \in \gr{v} \}
  \end{align*}
  is zero. However, if $\phi \in E_\lambda$ is annihilated by every $\LL \in \gr{v}$ then it is also annihilated by every local section of $\VV$ (recall that $\gr{v}$ contains a global frame $\LL_1, \ldots, \LL_n$ for $\VV$) -- that is, $\phi$ is a null solution of $\VV$. But since $\VV$ is elliptic (hence hypocomplex, see~\cite[Section~III.5]{treves_has}) and $G$ is compact and connected, any globally defined null solution of $\VV$ is constant, and thus $\phi \in E_0$. But if $\lambda$ is nonzero then $E_\lambda \cap E_0 = 0$.
\end{proof}

\def\cprime{$'$}

\end{document}